\documentclass[reqno, 12pt]{amsart}
\def\ssign{\textsection\nobreak\hspace{1pt plus 0.3pt}}
\makeatletter
\let\origsection=\section 
\def\mysection{\@mystartsection{section}{1}\z@{.7\linespacing\@plus\linespacing}{.5\linespacing}{\normalfont\scshape\centering\ssign}}
\def\section{\@ifstar{\origsection*}{\mysection}}
\def\appendix{\par\c@section\z@ \c@subsection\z@
   \let\sectionname\appendixname
   \let\section=\origsection 
   \def\thesection{\@Alph\c@section}}
\def\@mystartsection#1#2#3#4#5#6{\if@noskipsec \leavevmode \fi
 \par \@tempskipa #4\relax
 \@afterindenttrue
 \ifdim \@tempskipa <\z@ \@tempskipa -\@tempskipa \@afterindentfalse\fi
 \if@nobreak \everypar{}\else
     \addpenalty\@secpenalty\addvspace\@tempskipa\fi
 \@dblarg{\@mysect{#1}{#2}{#3}{#4}{#5}{#6}}}
\def\@mysect#1#2#3#4#5#6[#7]#8{\edef\@toclevel{\ifnum#2=\@m 0\else\number#2\fi}\ifnum #2>\c@secnumdepth \let\@secnumber\@empty
  \else \@xp\let\@xp\@secnumber\csname the#1\endcsname\fi
  \@tempskipa #5\relax
  \ifnum #2>\c@secnumdepth
    \let\@svsec\@empty
  \else
    \refstepcounter{#1}\edef\@secnumpunct{\ifdim\@tempskipa>\z@ \@ifnotempty{#8}{\@nx\enspace}\else
        \@ifempty{#8}{.}{\@nx\enspace}\fi
    }\@ifempty{#8}{\ifnum #2=\tw@ \def\@secnumfont{\bfseries}\fi}{}\protected@edef\@svsec{\ifnum#2<\@m
        \@ifundefined{#1name}{}{\ignorespaces\csname #1name\endcsname\space
        }\fi
      \@seccntformat{#1}}\fi
  \ifdim \@tempskipa>\z@ \begingroup #6\relax
    \@hangfrom{\hskip #3\relax\@svsec}{\interlinepenalty\@M #8\par}\endgroup
    \ifnum#2>\@m \else \@tocwrite{#1}{#8}\fi
  \else
  \def\@svsechd{#6\hskip #3\@svsec
    \@ifnotempty{#8}{\ignorespaces#8\unskip
       \@addpunct.}\ifnum#2>\@m \else \@tocwrite{#1}{#8}\fi
  }\fi
  \global\@nobreaktrue
  \@xsect{#5}}
\makeatother

\usepackage{amsmath,amssymb,amsthm}
\usepackage{mathrsfs} 
\usepackage{mathabx}\changenotsign
\usepackage{dsfont}

\usepackage{xcolor}
\usepackage[backref]{hyperref} 
\hypersetup{
	colorlinks,
	linkcolor={red!60!black},
	citecolor={green!60!black},
	urlcolor={blue!60!black}
}

\usepackage[open,openlevel=2,atend]{bookmark}

\usepackage[abbrev,msc-links,backrefs]{amsrefs}

\usepackage{doi}

\renewcommand{\PrintDOI}[1]{\doi{#1}}

\usepackage[T1]{fontenc}
\usepackage{lmodern}
\usepackage[babel]{microtype}
\usepackage[english]{babel}

\usepackage{geometry}
\numberwithin{equation}{section}
\numberwithin{figure}{section}

\usepackage{enumitem}
\def\rmlabel{\upshape({\itshape \roman*\,})}

\def\alabel{\upshape({\itshape \alph*\,})}
\def\Alabel{\upshape({\itshape \Alph*\,})}

\renewcommand\labelenumi{(\roman{enumi})}
\renewcommand\theenumi\labelenumi

\theoremstyle{plain}
\newtheorem{thm}{Theorem}[section]

\newtheorem{claim}[thm]{Claim}

\newtheorem{cor}[thm]{Corollary}
\newtheorem{lemma}[thm]{Lemma}

\theoremstyle{definition}
\newtheorem{rem}[thm]{Remark}

\let\eps=\varepsilon
\let\theta=\vartheta
\let\rho=\varrho
\let\phi=\varphi

\def\NN{\mathds N}
 
\def\PP{\mathds P}

\def\RR{\mathds R}

\def\ind{\mathds 1}

\def\cU{{\mathcal U}}

\def\ff{\mathfrak{f}}

\def\hp{\hat p}

\let\polishlcross=\l
\def\l{\ifmmode\ell\else\polishlcross\fi}

\makeatletter
\def\moverlay{\mathpalette\mov@rlay}
\def\mov@rlay#1#2{\leavevmode\vtop{   \baselineskip\z@skip \lineskiplimit-\maxdimen
		\ialign{\hfil$\m@th#1##$\hfil\cr#2\crcr}}}
\newcommand{\charfusion}[3][\mathord]{
	#1{\ifx#1\mathop\vphantom{#2}\fi\mathpalette\mov@rlay{#2\cr#3}}
	\ifx#1\mathop\expandafter\displaylimits\fi
}
\makeatother

\newcommand{\dcup}{\charfusion[\mathbin]{\cup}{\cdot}}

\def\tand{\ \text{and}\ }
\def\qand{\quad\text{and}\quad}
\def\qqand{\qquad\text{and}\qquad}

\newcommand{\vrhup}[1]{\scaleobj{0.6}{\scalerel*{\rightharpoonup}{#1}}}
\newcommand{\nrhup}{\mathord{\scaleobj{0.6}{\scalerel*{\rightharpoonup}{x}}}}
\newcommand{\wrhup}{\scaleobj{0.6}{\scalerel*{\rightharpoonup}{W}}}
\def\vseq#1{\ThisStyle{  \mathord{\vbox{\offinterlineskip\ialign{    \hfil##\hfil\cr
					$\SavedStyle{}_{\smash{\vrhup#1}}$\cr
					\noalign{\kern-0.7\scriptspace}
					$\SavedStyle#1$\cr}}}}}
\def\seq#1{\ThisStyle{  \mathord{\vbox{\offinterlineskip\ialign{    \hfil##\hfil\cr
					$\SavedStyle{}_{\smash{\nrhup}}$\cr
					\noalign{\kern-0.5\scriptspace}
					$\SavedStyle#1$\cr}}}}}
\def\wseq#1{\ThisStyle{  \mathord{\vbox{\offinterlineskip\ialign{    \hfil##\hfil\cr
					$\SavedStyle{}_{\smash{\wrhup#1}}$\cr
					\noalign{\kern-0.7\scriptspace}
					$\SavedStyle#1$\cr}}}}}

\let\setminus=\smallsetminus

\let\lra=\longrightarrow
\let\to=\lra

\makeatletter
\newcommand{\oset}[3][0ex]{\mathrel{\mathop{#3}\limits^{
			\vbox to#1{\kern-2.1\ex@
				\hbox{$\scriptstyle#2$}\vss}}}}
\makeatother

\DeclareMathSymbol{*}{\mathbin}{symbols}{"03}
\def\clra{\oset{*\hspace{1pt}}{\lra}}

\makeatletter
\newcommand{\pushright}[1]{\ifmeasuring@#1\else\omit\hfill$\displaystyle#1$\fi\ignorespaces}
\newcommand{\pushleft}[1]{\ifmeasuring@#1\else\omit$\displaystyle#1$\hfill\fi\ignorespaces}
\makeatother

\usepackage{tikz}
\usepackage{scalerel}
\newsavebox\eebox
\savebox\eebox{\tikz{
		\draw[black,fill=black] (90:1) circle (.35);
		\draw[black,fill=black] (210:1) circle (.35);
		\draw[black,fill=black] (330:1) circle (.35);
		\draw[black,line width=0.28cm ] (90:1) -- (330:1);
		\draw[black,line width=0.28cm ] (90:1) -- (210:1);
		\draw[opacity=0] (0:1.2) circle (0.1);
}}
\newcommand{\ee}{\mathord{\scaleobj{1.2}{\scalerel*{\usebox{\eebox}}{x}}}}
\newsavebox\epbox
\savebox\epbox{\tikz{
		\draw[black,fill=black] (0.866,1) circle (.35);
		\draw[black,fill=black] (-0.866,1) circle (.35);
		\draw[black,fill=black] (210:1) circle (.35);
		\draw[black,fill=black] (330:1) circle (.35);
		\draw[black,line width=0.28cm ] (0.866,1) -- (330:1);
		\draw[black,line width=0.28cm ] (-0.866,1) -- (210:1);
		\draw[opacity=0] (0:1.2) circle (0.1);
}}
\newcommand{\ep}{\mathord{\scaleobj{1.2}{\scalerel*{\usebox{\epbox}}{x}}}}

\let\N=\NN
\let\R=\RR
\let\df=\fd
\let\Hc=\cH
\let\Kc=\cK
\let\Pc=\cP
\let\Jn=J
\let\pst=\ps

\newcommand{\Gnp}{G(n,p)}
\newcommand{\Gnpst}{G(n, \pst)}
\newcommand{\npl}{p^{\binom \ell 2}n^{\ell}}

\newcommand{\norm}[1]{\left \|#1 \right \|}
\newcommand{\cutnorm}[1]{\left \|#1 \right \|_\square}

\begin{document}
	\
	\title{Canonical colourings in random graphs}

	\author[N.~Kam\v{c}ev]{Nina Kam\v{c}ev}
	\address{Department of Mathematics, Faculty of Science, University of Zagreb, Croatia}
	\email{nina.kamcev@math.hr}

	\author[M.~Schacht]{Mathias Schacht}
	\address{Fachbereich Mathematik, Universit\"at Hamburg, Hamburg, Germany}
	\email{schacht@math.uni-hamburg.de}
	
	\thanks{The research was supported through the European Union's Horizon~2020 research and innovation programme under 
		the Marie Sk\l odowska-Curie Action RobSparseRand 101038085 and by the ERC Consolidator Grant PEPCo~724903.}

	
	\dedicatory{Dedicated to the memory of Martin Aigner}

	\begin{abstract}
		R\"odl and Ruci\'nski~[\emph{Threshold functions for Ramsey properties}, J.~Amer. Math. Soc.~\textbf{8}~(1995)] established 
		Ramsey's theorem for random graphs. In particular, for fixed integers $r$, $\ell\geq 2$ they proved that 
		$\hp_{K_\l,r}(n)=n^{-\frac{2}{\l+1}}$ is a threshold for the Ramsey property that every $r$-colouring of the 
		edges of the binomial random graph $G(n,p)$ yields 
		a monochromatic copy of $K_\ell$. 
		We investigate how this result extends to arbitrary colourings of $G(n,p)$ with an unbounded number of colours. In this context,
		Erd\H{o}s and Rado~[\emph{A combinatorial theorem}, J.~London Math. Soc.~\textbf{25}~(1950)] proved that any edge-colouring of a sufficiently large complete graph contains  one of four canonical colourings of $K_\ell$ -- a monochromatic, or rainbow, or \textit{min}  or \textit{max}  colouring; a \textit{min-colouring} of $K_\ell$ is a colouring in which two edges have the same colour if and only if they have the same minimal vertex.
		We transfer the Erd\H os--Rado theorem to the random graph $G(n,p)$ and show that both thresholds coincide when $\l\geq 4$.
		As a consequence, the proof yields $K_{\l+1}$-free graphs~$G$ for which every edge colouring contains a 
		canonically coloured $K_\l$.
		
		The $0$-statement of the threshold is a direct consequence of the corresponding statement of the R\"odl--Ruci\'nski theorem and the 
		main contribution is the $1$-statement. The proof of the $1$-statement employs the transference principle of 
		Conlon and Gowers~[\emph{Combinatorial theorems in sparse random sets}, Ann. of Math.~(2)~\textbf{184}~(2016)].
	\end{abstract}

	\maketitle

	\section{Introduction}
	\label{sec:introduction}
	In the last three decades, extremal and Ramsey-type properties of random graphs were considered, 
	which led to several general approaches to these questions (see, e.g.~\cites{bms15,cg16,klr97,nenadov21,rr95,st15,schacht16} 
	and references therein). We consider Ramsey-type questions for the binomial random graph~$G(n,p)$. For graphs $G$ and $H$
	and an integer $r\geq 2$,
	we write 
	\[
	G\lra(H)_r
	\]
	to signify the statement that every $r$-colouring of the edges of $G$ yields a monochromatic copy of~$H$. Ramsey's theorem~\cite{r30} 
	asserts that for fixed $H$ and $r$ the family of graphs~$G$ satisfying $G\lra(H)_r$ is non-empty. Obviously, this family is monotone\footnote{that is, closed with respect to adding edges}. Hence, by a well-known result of Bollob\'as and Thomason~\cite{bt87}, there is a threshold function $\hp_{H,r}\colon \NN\to[0,1]$ such that  
	\begin{equation}\label{eq:threshold}
		\lim_{n\to\infty}\PP\big(G(n,p)\lra(H)_r\big)=\begin{cases}
			0\,,&\text{if $p\ll\hp_{H,r}$,}\\
			1\,,&\text{if $p\gg\hp_{H,r}$.}
		\end{cases}
	\end{equation}
	As usual we shall refer to any such function as \textit{the} threshold of that property, even though it is not unique.
	
	R\"odl and Ruci\'nski~\cites{rr93,rr95} determined the threshold $\hp_{H,r}$ for every graph $H$ and every fixed number of colours~$r$. We restrict our attention
	to the situation when $H$ is a clique $K_\l$ and state their result for that case only.
	\begin{thm}[R\"odl \& Ruci\'nski]
		\label{thm:random-ramsey}
		For every $r\geq 2$ and $\l\geq 3$, the equality $\hp_{K_\l,r}(n)=n^{-\frac{2}{\ell+1}}$ holds.\qed
	\end{thm}
	In fact, R\"odl and Ruci\'nski established a \emph{semi-sharp} threshold, i.e., the $0$-statement in~\eqref{eq:threshold} holds as long as 
	$p(n)\leq c_{\l,r}n^{-\frac{2}{\ell+1}}$ for some sufficiently small constant~$c_{\l,r}>0$
	and, similarly,  the $1$-statement becomes true already if~$p(n)\geq C_{\l,r}n^{-\frac{2}{\ell+1}}$ for some $C_{\l,r}$.
	This was sharpened recently in~\cite{fkss}, where the gap between $c_{\l,r}$ and $C_{\l,r}$ was closed.
	Perhaps surprisingly, the asymptotic growth of the threshold function $\hp_{K_\l,r}(n)$ in Theorem~\ref{thm:random-ramsey}
	is independent of the number of colours~$r$. 
	
	We are interested in edge colourings of $\Gnp$ which are not restricted to a fixed number of colours. 
	However, if the number of colours is unrestricted, 
	then this allows injective edge colourings and, consequently, monochromatic $K_\ell$-copies might be prevented. Nevertheless,
	Erd\H os and Rado~\cite{er50}, showed that certain \emph{canonical} patterns\footnote{A \textit{pattern} refers to a colouring of the `target graph' $K_\ell$.} are unavoidable in edge colourings of sufficiently large 
	cliques. 
	Obviously, the monochromatic and the injective pattern (in which each edge receives a different colour) must be canonical.
	Two additional canonical patterns arise by ordering the vertices of~$K_n$ and colouring every edge~$uv$ 
	by $\min\{u,v\}$ or colouring every edge by its maximal vertex. More generally, for finite graphs $G$ and $H$ 
	with ordered vertex sets, we write 
	\[
	G\clra(H)
	\]
	if for every edge colouring $\phi\colon E(G)\to\NN$ there exists an order-preserving graph embedding~$\varsigma\colon H\to G$ such that one of the 
	following holds:
	\begin{enumerate}[label=\alabel]
		\item\label{it:mono} the copy $\varsigma(H)$ of $H$ is monochromatic under $\phi$, 
		\item\label{it:rain} or $\phi$ restricted to $E(\varsigma(H))$ is injective, 	
		\item\label{it:min} or for all edges $e$, $e'\in E(\varsigma(H))$ we have 
		$\phi(e)=\phi(e')\ \Longleftrightarrow\ \min(e)=\min(e')$,
		\item\label{it:max} or for all edges $e$, $e'\in E(\varsigma(H))$ we have 
		$\phi(e)=\phi(e')\ \Longleftrightarrow\ \max(e)=\max(e')$.
	\end{enumerate}
	We call an ordered copy of $H$ in $G$  \emph{canonical} (with respect to $\varphi$) if it displays one of the four patterns described in~\ref{it:mono}--\ref{it:max}.
	
	Note that for the patterns described in~\ref{it:mono} and~\ref{it:rain} the orderings of the vertex sets have no bearing. We shall 
	refer to copies exhibiting an injective colouring as \emph{rainbow} copies of $H$ (even if $|E(H)|\neq 7$). Moreover, we refer to  
	the patterns appearing in~\ref{it:min} and~\ref{it:max} as \emph{min-coloured} and \emph{max-coloured}, respectively. In case only the 
	backward implications in~\ref{it:min} or~\ref{it:max} are enforced, then we refer to those colourings as \emph{non-strict}, e.g., 
	if $\min(e)=\min(e')$ yields $\phi(e)=\phi(e')$ for all edges $e$, $e'\in E(\varsigma(H))$, then~$\varsigma(H)$ is a non-strictly min-coloured copy 
	of~$H$. Obviously, monochromatic copies are also non-strictly min- and max-coloured. 
	
	Hereafter, the vertex sets of all 
	graphs considered are ordered. In particular, for cliques and random graphs we assume 
	\[
	V(K_n)=[n]
	\qqand
	V(G(n,p))=[n]\,.
	\]
	With this notation at hand, the aforementioned 
	canonical Ramsey theorem of Erd\H os and Rado~\cite{er50} restricted to the graph case asserts that canonical copies are unavoidable.
	\begin{thm}[Erd\H os \& Rado]\label{thm:canonical-ramsey}
		For every $\l\geq 3$, there exists $n$ such that $K_{n}\clra(K_{\l})$.\qed
	\end{thm}

	We are interested in a common generalisation of Theorems~\ref{thm:random-ramsey} and~\ref{thm:canonical-ramsey}. 
	Namely, owing to Theorem~\ref{thm:canonical-ramsey}, for any ordered graph $H$ and sufficiently large $n$, the fact that $K_n\clra(H)$ raises 
	the problem of estimating 
	the threshold $\hp_H\colon \NN\to[0,1]$ such that  
	\begin{equation}\label{eq:cthreshold}
		\lim_{n\to\infty}\PP\big(G(n,p)\clra(H)\big)=\begin{cases}
			0\,,&\text{if $p\ll\hp_{H}$,}\\
			1\,,&\text{if $p\gg\hp_{H}$.}
		\end{cases}
	\end{equation}
    Unless the vertices of a graph $H$ can be covered by only two edges, the only canonical colourings of $H$ using at most two colours are monochromatic. Consequently, 
	\[
	\hp_H\geq\hp_{H,2}
	\] 
	for every such graph~$H$; indeed, a two-colouring of $G(n,p)$ establishing a lower bound on $\hp_{H,2}$ automatically yields a lower bound on $\hp_H$. In particular, for cliques on at least four vertices, this may suggest that the asymptotics of $\hp_{K_\ell, 2}$ and $\hp_{K_\ell}$ coincide, and our main result verifies this. 
	\begin{thm} \label{t:main}
		For every $\l \geq 4$, there exists $C>0$ such that for $p=p(n)\geq Cn^{-\frac{2}{\l+1}}$ we have 
		\[
		\lim_{n\to\infty}\PP\big(G(n,p)\clra(K_\l)\big)
		=
		1\,.
		\]
	\end{thm}
	Combining Theorem~\ref{t:main} with the corresponding lower bound on $\hp_{K_\l,2}$ shows that the threshold for 
	the canonical Ramsey property is semi-sharp for $\l\geq 4$. For $\l=3$, we recall that the canonical Ramsey threshold 
	is indeed smaller than the Ramsey threshold~$n^{-1/2}$. In fact, one can check that every edge colouring 
	of~$K_4$ yields a canonical copy of the triangle and, hence, $\hp_{K_3}\leq n^{-2/3}$.  We are unaware of the corresponding lower bound on $\hp_{K_3}$. 
	
	Moreover, we note that for $p=O(n^{-\frac{2}{\l+1}})$
	the random graph~$G(n,p)$ is likely to contain only $o(pn^2)$ cliques $K_{\l+1}$. 
	In the proof of Theorem~\ref{t:main} we can delete an edge from each 
	such clique. Consequently, we obtain the following statement in structural Ramsey theory, which can be viewed as a
	Folkman-type extension of the Erd\H os--Rado theorem for graphs.
	\begin{cor}\label{c:structure}
		For every $\l \geq 4$ there exists a $K_{\l+1}$-free graph~$G$ such that $G\clra(K_\l)$.
		Moreover, $G$ contains no two distinct copies of $K_\l$ that share at least three vertices.
	\end{cor}
	In the context of Ramsey's theorem, the existence of such a graph~$G$  was asked for by Erd\H os and Hajnal~\cite{eh67}; for two colours
	this was established by Folkman~\cite{f70}, and for any fixed number of colours by Ne\v set\v ril and R\"odl~\cite{nr76}.
	The graph~$G$ in Corollary~\ref{c:structure} will be obtained by modifying the random graph and, hence, the proof is non-constructive.
	Reiher and R\"odl~\cite{rr} pointed out 
	that the first part of Corollary~\ref{c:structure} can also be proved in a constructive manner
	by means of the \emph{partite construction method} of Ne\v set\v ril and R\"odl~\cite{nr81}.
	While this approach falls short to exclude $K_\l$'s intersecting in triangles, it has the 
	advantage that it readily extends to $k$-uniform hypergraphs for every $k\geq 3$.

	We conclude this introduction with a short overview of the main ideas of the proof of Theorem~\ref{t:main}.
	Roughly speaking, the proof is inspired by the proof of the canonical graph Ramsey theorem laid out by  
	Lefmann and R\"odl~\cite{lr95} and Alon et al.~\cite{ajmp03}.
	This approach pivots on a case distinction of the edge colouring of the 
	underlying graph $K_n$. The first case, when many different colours appear everywhere, which is captured by assuming 
	that every vertex is incident to only $o(n)$ edges of the same colour, leads to rainbow copies of~$K_\l$.
	In the other case, there is a vertex with a monochromatic neighbourhood of size~$\Omega(n)$, which by iterated applications, 
	as in the standard proof of Ramsey's theorem, leads to a non-strictly min- or max-coloured $K_{(\l-2)^2+2}$. Such a non-strictly min/max-coloured clique 
	contains a canonical~$K_\l$ by a straightforward application of Dirichlet's box principle. 
	
	Transferring such an approach from $K_n$ to $G(n,p)$ for $p=O(n^{-\frac{2}{\l+1}})$ faces several challenges. Firstly, we shall not use a 
	$K_{(\l-2)^2+2}$ in the coloured host graph, as such large cliques are extremely unlikely to appear in $G(n,p)$
	for that edge probability.  Moreover, in the more challenging second case, when the colouring is unbounded, it 
	is certainly not sufficient to consider one vertex with a large 
	monochromatic neighbourhood (of size $\Omega(pn)$), as again, this neighbourhood is too sparse to contain any useful structure in $\Gnp$. Thus we resort to a robust 
	version of the above-mentioned argument, building a large non-strictly min- or max-coloured subgraph which contains $\Omega(n^2p)$ edges.

	The bounded case, with at most $\lambda$ edges of every colour incident to any given vertex of~$\Gnp$, is a problem of independent interest.
	For example,  $\lambda=1$ corresponds to studying proper edge colourings 
	of $G(n,p)$ and anti-Ramsey properties (see, e.g.~\cites{kkm14,npss17,kmps19} and the references therein). In fact, for $\ell \geq 5$, there 
	are proper colourings of $\Gnp$ with~$p = cn^{-\frac{2}{\ell+1}}$ which do not contain a rainbow copy of $K_{\ell}$ (see~\cite{kmps19}), which 
	is an alternative argument for~$\hat p_{K_\ell} \geq cn^{-\frac{2}{\ell+1}}$ and another obstruction for Theorem~\ref{t:main}. 
	For the proof of Theorem~\ref{t:main} presented here, we will need to guarantee rainbow copies of $K_\l$ under the weaker assumption that $\lambda = o(pn)$. 
	This can be viewed as a partial
	extension of the work of Kohayakawa, Kostadinidis, and Mota~\cite{kkm14}. In both cases (bounded and unbounded colourings)
	the transference principle for random discrete structures developed by Conlon and Gowers~\cite{cg16} is an integral part of the proof.
	
	\subsection*{Organisation}
	
	In the next section, we present the two main 
	lemmata rendering the case distinction sketched above, and deduce Theorem~\ref{t:main}. The proofs of these lemmata are deferred to Sections~\ref{sec:rainbow} 
	and~\ref{sec:strictly-min}, along with the corresponding preliminaries.
	We conclude with a discussion of possible generalisations of this work from cliques~$K_\l$ to general graphs~$H$ and of related open problems in Section~\ref{sec:conclusion}.

	\section{Proof of the main result}
	\label{sec:outline}
	
	\subsection{Proof of the canonical Ramsey theorem for graphs}
	The proof of Theorem~\ref{t:main} adopts some ideas of the canonical Ramsey theorem for graphs from the 
	work of Lefmann and R\"odl~\cite{lr95} and Alon et al.~\cite{ajmp03} and below we recall their argument.
	For $\l\geq 3$ we fix
	\begin{equation}\label{eq:cRamsey-det}
		\delta=\frac{1}{4\l^{3}}
		\qand
		n\geq 2^{6\l^2(\log_2(\l)+1)}
	\end{equation}
	and first we consider bounded colourings $\phi\colon E(K_n)\to\NN$. More precisely, we say such a colouring is \emph{$\delta$-bounded}
	if for every colour $c\in \NN$ and every vertex $v\in V(K_n)$ we have
	\begin{equation}
    \label{eq:delta-bounded}
	d_c(v)=\big|N_c(v)\big|=\big|\{w\in V(K_n)\colon \phi(vw)=c\}\big|\leq \delta n\,.
	\end{equation} 
	Roughly speaking, bounded colourings have the property that many different colours are
	``present everywhere'' and this yields rainbow copies of $K_\l$.
	In fact, a simple counting argument shows that for $\delta$-bounded colourings, at most $\delta n^3/2$ triples contain two edges of the same colour
	and  at most $\delta n^4/8$ quadruples contain two disjoint edges of the same colour.
	Consequently, selecting every vertex of~$K_n$ independently with 
	probability~$2\l/n$ and removing a vertex from every such triple and every such quadruple,
	establishes the existence of $\l$ vertices inducing a rainbow $K_\l$.
	
	The second part of the proof resembles the standard proof of Ramsey's theorem for graphs 
	and iterates along large monochromatic neighbourhoods. Given the observation above for bounded colourings, 
	we 	may assume that the edge colouring $\phi$ is unbounded in a hereditary way (meaning that no induced subgraph of order $\Omega(n)$ satisfies~\eqref{eq:delta-bounded}) and this requires the exponential
	lower bound~\eqref{eq:cRamsey-det} on $n$.
	
	More precisely, assuming that $\phi$ fails to induce a rainbow copy of $K_\l$ gives rise to a vertex $v\in V(K_n)$, 
	a colour $c$, and a comparability sign $\diamond\in\{<\,,>\}$ such that 
	\[
	d^{\diamond}_{c}(v)
	=
	\big|N^{\diamond}_{c}(v)\big|
	=
	\big|\{w\in V(K_n)\colon \phi(vw)=c \tand v\diamond w\}\big| 
	> 
	\frac{\delta n}{2}\,.
	\]
	Restricting our attention to the colouring $\phi$ on the vertices contained in $N^{\diamond}_{c}(v)$ and iterating this argument $L=2(\l-2)^2+2$ times 
	leads to a sequence $(v_i,c_i,\diamond_i)_{i\in [L]}$ such that for every~$i\in[L]$ we have
	\begin{equation}\label{eq:iteration-dense}
		\bigg|\bigcap_{j=1}^{i}N^{\diamond_j}_{c_j}(v_j)\bigg|
		>
		\Big(\frac{\delta}{2}\Big)^i n\,.
	\end{equation}
	In fact, owing to the choices in~\eqref{eq:cRamsey-det} we can iterate this step~$L$ times. 
	
	Furthermore, we may assume that there are indices 
	$1\leq i_0<\dots<i_{(\l-2)^2}< L$ such that~$\diamond_{i_j}$ is $<$ for all~$j$. Consequently, 
	the correspondingly indexed vertices $v_{i_0},\dots,v_{i_{(\l-2)^2}}$ together with $v_L$ 
	induce a non-strictly min-coloured
	clique on $(\l-2)^2+2$ vertices. Finally, if one of the colours appears $\l-1$ times among 
	$c_{i_0},\dots,c_{i_{(\l-2)^2}}$, then this yields a monochromatic~$K_\l$  among $v_{i_0},\dots,v_{i_{(\l-2)^2}}$, and $v_L$.
	Otherwise, at least $\l-1$ distinct colours appear and we are guaranteed to find a min-coloured~$K_\l$ instead.
	
	\subsection{Bounded and unbounded colourings in random graphs}
	For the proof of Theorem~\ref{t:main} we derive appropriate random versions 
	of the facts above that analyse bounded and unbounded colourings of~$G(n,p)$
	(see Lemmata~\ref{l:rainbow} and~\ref{l:strictly-min} below). We begin by defining a notion of boundedness central to our proof.
	Roughly speaking, an edge colouring of $G(n,p)$ is bounded if at most $o(pn)$ edges of the same colour are incident to any given vertex. 
	Similar to the proof in the deterministic setting, it will be useful to define this property for large subsets of vertices, which 
	is made precise as follows. 
	
	Given a graph $G=(V,E)$ with an edge colouring $\phi\colon E\to\NN$, 
	a subset $U\subseteq V$, and reals $\delta>0$, $p\in(0,1]$ we say $\phi$ is \emph{$(\delta,p)$-bounded on $U$} if 
	for every colour $c\in \NN$ and every vertex $u\in U$ we have 
	\[
	d_c(u,U)=\big|N_c(u,U)\big|=\big|\{w\in U\colon \phi(uw)=c\}\big|\leq \delta p|U|\,.
	\]
	Lemma~\ref{l:rainbow} (stated below) asserts that bounded edge colourings of $G(n,p)$ for $p\gg n^{-\frac{2}{\l+1}}$ 
	yield rainbow copies of $K_\l$ asymptotically almost surely, i.e., with probability tending to $1$ as $n\to\infty$. 
	
	In view of Corollary~\ref{c:structure}, we define the \emph{$\l$-clean} subgraph $G_\l$ of a given graph $G$ on $[n]$ 
	as follows: Consider all edges of $G$ in lexicographic order and remove an edge $e$ in the current subgraph of $G$, if the edge $e$
	is contained in the intersection of two distinct $K_\l$-copies sharing at least three vertices. Actually, the precise ordering is not 
	relevant for our argument, but the uniqueness of the $\l$-clean subgraph $G_\l\subseteq G$ defined above will be convenient. Note that $G_\l$ contains no copy of $K_{\l+1}$, since this would yield two $K_\l$'s intersecting in $\l-1$ vertices.

	\begin{lemma} 
		\label{l:rainbow} 
		For all integers $\l\geq 4$ and every $\nu>0$ there is some constant $C>0$ such that for 
		$p=p(n)\geq Cn^{-\frac{2}{\l+1}}$ asymptotically almost surely the following holds for $G\in G(n,p)$.
		
		If $\phi\colon E(G)\to \NN$ is $(\l^{-5}/4,p)$-bounded on some $U\subseteq V(G)$ with $|U|\geq \nu n$, 
		then $U$ induces a rainbow copy of $K_\l$ in $G$.
		
		Moreover, if in addition we have $p(n)\leq n^{-\frac{2\l-5}{\l^2-\l-4}}/\omega(n)$ for some arbitrary function~$\omega$
		tending to infinity
		as $n \to \infty$, then the $\l$-clean subgraph $G_\l\subseteq G$ also contains a rainbow copy of~$K_\l$. 
	\end{lemma}
	Lemma~\ref{l:rainbow} strengthens a result of Kohayakawa, Kostadinidis, and Mota~\cite{kkm14}, where a more restrictive boundedness assumption
	is required. The proof of Lemma~\ref{l:rainbow} makes use of the transference principle of Conlon and Gowers~\cite{cg16}, which 
	allows us to transfer the bounded case in the deterministic setting to the random environment. We defer the proof of Lemma~\ref{l:rainbow} 
	to Section~\ref{sec:rainbow}. The second lemma yields canonical copies in unbounded colourings.

	\begin{lemma}
		\label{l:strictly-min}
		For all integers $\l\geq 3$ and every $\delta>0$ there is some constant $C>0$ such that for 
		$p=p(n)\geq Cn^{-\frac{2}{\l+1}}$ asymptotically almost surely the following holds for $G\in G(n,p)$.
		
		If $\phi\colon E(G)\to \NN$
		has the property that every $U\subseteq V(G)$ with size $|U|\geq \delta^{5 \l^2} n$
		satisfies
		\begin{equation}\label{eq:delta-unbounded-undir}
			\big|\{u \in U\colon d_c(u, U) \geq 8\delta p |U| \text{ for some colour } c\} \big|
			\geq 
			\frac{|U|}{2}\,,
		\end{equation}
		then $G$ contains a canonical copy of $K_\l$.
		
		Moreover, if in addition we have $p(n)\leq n^{-\frac{2\l-5}{\l^2-\l-4}}/\omega(n)$ for some arbitrary function~$\omega$ 
		tending to infinity
		as $n \to \infty$, then the $\l$-clean subgraph $G_\l\subseteq G$ also contains a rainbow copy of~$K_\l$.
	\end{lemma}
	As in the unbounded case in the deterministic setting, the proof of Lemma~\ref{l:strictly-min} yields
	either a monochromatic, or a min-coloured, or a max-coloured copy of $K_\l$. The proof of Lemma~\ref{l:strictly-min} is more 
	involved and we give a detailed outline in Section~\ref{sec:outline-unbounded}.
	We conclude this section with the short proof of Theorem~\ref{t:main} and Corollary~\ref{c:structure}
	based on Lemmata~\ref{l:rainbow} and~\ref{l:strictly-min}.
	
	\begin{proof}[Proof of Theorem~\ref{t:main} and Corollary~\ref{c:structure}]
		Given $\l\geq 4$ we set $\delta=\l^{-5}/64$ and $\nu=\delta^{5 \l^2}/2$
		and let $C$ be sufficiently large so that we can appeal to Lemma~\ref{l:rainbow} with $\l$ and $\nu$
		and to Lemma~\ref{l:strictly-min} with~$\l$ and~$\delta$.
		Owing to the monotonicity of the canonical Ramsey property, for the proof of Theorem~\ref{t:main} we may assume 
		$p=p(n)=Cn^{-\frac{2}{\l+1}}$. Since $\l\geq 4$ this implies $p(n)\leq n^{-\frac{2\l-5}{\l^2-\l-4}}/\omega(n)$ for some function $\omega$
		tending to infinity with~$n$.

		Let $G\in G(n,p)$ satisfy the conclusion of both lemmata
		and consider an arbitrary edge colouring  $\phi\colon E(G)\to \NN$ 
		of~$G$.

		For every $U\subseteq V(G)$ we consider its subset of \textit{unbounded} vertices in $U$
		\[
		B(U)=\{w\in U\colon d_c(w, U) \geq 8\delta p |U| \text{ for some colour } c\}\,.
		\]
		Owing to Lemma~\ref{l:strictly-min} we may assume that there is a set $U\subseteq V(G)$ satisfying
		$|U|\geq \delta^{5 \l^2} n$ and $|B(U)| < |U|/2$.
		Removing the unbounded vertices from~$U$ we arrive at a set
		\[
		U' 
		= 
		U \setminus B(U)
		\quad\text{with}\quad
		|U'|
		>
		\frac{|U|}{2}
		\geq 
		\nu n\,.
		\] 
		For every vertex $u \in U'$ and every colour $c$ we have
		\[
		d_c(u, U') 
		\leq 
		d_c(u, U) 
		<
		8\delta p |U| 
		<
		16\delta p |U'|\,.
		\]
		In other words, $\phi$ is $(16\delta,p)$-bounded on $U'$ and Lemma~\ref{l:rainbow} yields asymptotically almost surely a 
		rainbow copy of~$K_\ell$ in the $\l$-clean subgraph $G_\l\subseteq G$.
	\end{proof}

	\section{Rainbow cliques in bounded colourings of random graphs} 
	\label{sec:rainbow}
	We shall use the following notation.
	For a graph $G=(V,E)$, a vertex $v$, and a set $U\subseteq V$ we write $d_G(v,U)$ for the size of the neighbourhood of $v$ in $U$.
	For subsets $X$, $Y\subseteq V$ we denote by $e_G(X,Y)$ the number of edges with one vertex in $X$ and one vertex in $Y$, where edges in 
	$X\cap Y$ are counted twice, i.e., 
	\[	
	e_G(X,Y)=\big|\{(x,y)\in X\times Y\colon xy\in E\}\big|=\sum_{x\in X}d_G(x,Y)\,.
	\]
    In particular, $e_G(X) = e_G(X, X)$.
	Moreover, for some integer $\l\geq 2$ we denote by $\kappa_\l(G)$ the number of (labeled) copies of~$K_\l$ in $G$.
	For a family $\cU=(U_1,\dots,U_\l)$ of mutually disjoint vertex subsets of $V$
	we write $G[\cU]$ for the $\l$-partite subgraph induced by the sets $U_1,\dots,U_\l$.

	The proof of Lemma~\ref{l:rainbow} is based on the transference principle of Conlon and Gowers~\cite{cg16}, 
	which we use in the following form~\cite{cgss14}*{Theorem~3.2}.
	\begin{thm}[Conlon \& Gowers]
		\label{t:transference1}
		For all integers $\ell\geq 3$ and every $\eps>0$ there is some constant $C>0$ such that for $p=p(n)$ with $Cn^{-\frac{2}{\l+1}}\leq p\leq 1/C$ 
		and every $\zeta>0$ asymptotically almost surely the following holds for $G\in\Gnp$.
		
		For every family $\cU=(U_1,\dots,U_\l)$ of mutually disjoint vertex subsets of $V(G)$
		and every $\l$-partite subgraph~$S$ of $G[\cU]$, there exists an $\l$-partite  
		subgraph $D$ of $K_n[\cU]$ such that
		\begin{enumerate}[label=\rmlabel]
			\item\label{it:trans1-1} for all subsets $X$, $Y\subseteq V(G)$ we have $|e_S(X,Y)-p\cdot e_D(X,Y)|\leq \eps pn^2$
			\item\label{it:trans1-2} and $\big|\kappa_\l(S)-p^{\binom{\l}{2}}\cdot\kappa_\l(D)\big|\leq \eps p^{\binom{\l}{2}} n^\l$.
		\end{enumerate}
		Moreover, we have
		\begin{enumerate}[label=\rmlabel]\pushQED{\qed}
			\setcounter{enumi}{2}
			\item\label{it:trans1-1deg} for every $X\subseteq V(G)$ all but at most $\eps n$ vertices $v\in V(G)$ satisfy 
			\[
			|d_S(v,X)-p\cdot d_D(v,X)|\leq \eps pn
			\]
			\item\label{it:trans1-3} and if $G'$ is obtained from $G$ by removing at most $\zeta pn^2$ edges, then
			\[
			\big|\kappa_\l(G'[\cU])-p^{\binom{\l}{2}}|U_1|\cdots|U_\l|\big|\leq (\eps+\zeta) p^{\binom{\l}{2}} n^\l\,.\qedhere
			\]
		\end{enumerate}
	\end{thm}
	We say the graph $D$ provided by Theorem~\ref{t:transference1} is a \emph{dense model} for the subgraph $S$ of the sparse random graph. We remark that in~\cite{cgss14}, (i) is stated for disjoint subsets $X$ and $Y$, but the bound for all pairs of subsets follows easily from the inclusion-exclusion principle. 
	Furthermore,  
	the moreover-part is not stated in~\cite{cgss14}*{Theorem~3.2}. However, it easily follows from~\ref{it:trans1-1} 
	and~\ref{it:trans1-2} applied with an appropriately chosen $\eps'\ll\eps$. 
	In fact, part~\ref{it:trans1-1deg} follows from~\ref{it:trans1-1} applied to $X$ and 
	$Y^+$ being the set of vertices having too high degree in $D$, and a second application to $X$ and a similarly defined set $Y^-$ (see, e.g., proof of Lemma~\ref{l:degree-cutnorm} in Section~\ref{sec:prelim}).
	Part~\ref{it:trans1-3} can be deduced by applying~\ref{it:trans1-1} and~\ref{it:trans1-2} for $S'=G'[\cU]$. In fact, for this choice, 
	part~\ref{it:trans1-1}, combined with Chernoff's inequality  and the union bound over the choices for $\cU$,  
	implies that all $\binom{\l}{2}$ bipartite subgraphs~$D'[U_i,U_j]$ have density close to $1$. More precisely, 
	in this case $D'[\cU]$ and $K_n[\cU]$ differ by at most $(2\eps'+\zeta)n^2$ edges.
	Consequently,~$\big|\kappa_\l(D')-|U_1|\cdots|U_\l|\big|\leq (2\eps'+\zeta)n^\l$, which can be transferred to $S'=G'[\cU]$ by~\ref{it:trans1-2}.
	The two conclusions of the following lemma further strengthen the upper bound of part~\ref{it:trans1-3} and can be viewed as a customised version 
	of part~\ref{it:trans1-2} for our proof of Lemma~\ref{l:rainbow}.
	
	\begin{lemma} \label{l:upper-rainbow}
		For all integers $\ell\geq 4$ and every $\eps>0$ there is some constant $C>0$ such that for $p=p(n)$ with $Cn^{-1/m_2(K_\l)}\leq p\leq 1/C$ 
		asymptotically almost surely the following holds for $G\in\Gnp$.
		
		For every family $\cU=(U_1,\dots,U_\l)$ of mutually disjoint vertex subsets of $V(G)$
		and every $\l$-partite subgraph~$S$ of $G[\cU]$ the following holds:
		\begin{enumerate}[label=\alabel]
			\item \label{it:matching-bound} For $d_{12}=\frac{e_S(U_1, U_2) }{p|U_1||U_2|}$ 
			and $d_{34}=\frac{e_S(U_3, U_4) }{p|U_3||U_4|}$ we have
			\[
			\kappa_\l(S)
			\leq
			d_{12}d_{34}\cdot p^{\binom{\l}{2}}|U_1|\cdots|U_\l|+\eps p^{\binom{\l}{2}}n^\l\,.	
			\]
			\item \label{it:cherry-bound} 
			For $c_{123}=\frac{\sum_{u\in U_1}d_S(u,U_2)d_S(u,U_3)}{p^2|U_1||U_2||U_3|}$ we have
			\[
			\kappa_\l(S)
			\leq
			c_{123}\cdot p^{\binom{\l}{2}}|U_1|\cdots|U_\l|+\eps p^{\binom{\l}{2}}n^\l\,.	
			\]
		\end{enumerate}
	\end{lemma}
	\begin{proof}
		We only prove part~\ref{it:cherry-bound}, since the proof of~\ref{it:matching-bound} is very similar. Let $G \in G(n,p)$.
		Given $\l$ and $\eps \in \left(0, 2^{-9} \right)$, let $C$ be sufficiently large, so that Theorem~\ref{t:transference1} applies for $\l$ and $\eps/16$.
		
		For the given $\l$-partite subgraph $S$ on partition classes $U_1,\dots,U_\l$ we may assume, without loss of generality, 
		that $|U_i|\geq \eps n/2$, since otherwise the bound easily follows from~$\kappa_\l(S)\leq\kappa_\l(G)$ and part~\ref{it:trans1-3} 
		of Theorem~\ref{t:transference1}. Moreover, a standard argument using Chernoff's inequality and the union bound implies that
		\begin{equation}\label{eq:edeg}
			d_S(v)\leq d_G(v)\leq 2pn
		\end{equation}
		for every vertex $v\in V(G)$.
		
		Let~$D$ be the dense model of $S$ provided by Theorem~\ref{t:transference1}. In view of part~\ref{it:trans1-2} of Theorem~\ref{t:transference1},
		it suffices to show that
		\begin{equation} \label{eq:D-bound}
			\kappa_\l(D) 
			\leq 
			c_{123}\cdot|U_1|\cdots|U_\l|+ \frac{\eps}{2}n^\l\,.
		\end{equation}
		
		Let $Y^+$ be the set of vertices $u \in U_1$ for which  $p\cdot d_D(u, U_2)>d_S(u, U_2)+\eps pn/16$ 
		or~$p\cdot d_D(u, U_3)>d_S(u, U_3)+\eps pn/16$. Part~\ref{it:trans1-1deg} tells us $|Y^+|\leq \eps n/8$ and combined with~\eqref{eq:edeg}
		it follows
		\begin{align*}
			p^2\cdot \sum_{u \in U_1} d_D(u, U_2)d_D(u, U_3) 
			&\leq
			\sum_{u \in U_1} d_S(u, U_2)d_S(u, U_3) + \Big(\frac{\eps}{4}+\frac{\eps^2}{16}\Big) p^2n^3 + |Y^+|p^2n^2\\
			&\leq 
			c_{123}\cdot p^2 |U_1| |U_2| |U_3| + \frac{\eps}{2} p^2n^3\,.
		\end{align*}
		Consequently, the number of $K_{1,2}$ in $D$ with center vertex in $U_1$ and leaves in $U_2$ and $U_3$ is bounded from above by 
		\[
		c_{123}\cdot |U_1| |U_2| |U_3| + \frac{\eps}{2}n^3
		\]
		and~\eqref{eq:D-bound}  is obtained by bounding the extension of each of these $K_{1,2}$ trivially by 
		$|U_4| \cdots|U_\l|$.
	\end{proof}
	
	We conclude this section with the proof of Lemma~\ref{l:rainbow}, which yields rainbow cliques in bounded colourings of the random graph.
	
	\begin{proof}[Proof of Lemma~\ref{l:rainbow}]
		Given $\l\geq 4$ and $\nu>0$ we define the auxiliary constant $\delta=\l^{-4}/2$ and 
		let $C$ be sufficiently large so that the Theorem~\ref{t:transference1} and Lemma~\ref{l:upper-rainbow} apply for 
		\begin{equation}\label{eq:eps-bdd}
			\eps
			=
			\Big(\frac{\nu}{2\l}\Big)^\l\,.
		\end{equation}
		Suppose $G\in \Gnp$ satisfies the conclusions of Theorem~\ref{t:transference1} and Lemma~\ref{l:upper-rainbow}.
		We may also assume that for every subset $X\subseteq V(G)$ of size at least $|X|\geq \frac{\nu n}{2\l}$
		we have $d_G(v,X)\leq 1.1p|X|$ for all but at most $n/\log(n)$ vertices $v$.
		
		Below we only prove the moreover-part of the lemma for the $\l$-clean subgraph $G_\l\subseteq G$, since the proof 
		for~$G$ without the upper bound on~$p$  is identical. Hence, we assume $p=p(n)\leq n^{-\frac{2\l-5}{\l^2-\l-4}}/\omega(n)$ for some function 
		$\omega$ tending to infinity. From this upper bound on~$p$ it follows by Markov's inequality that,  asymptotically almost 
		surely, the number of distinct pairs of $K_\l$ sharing more than two vertices is at most~$o(pn^2)$
		and, therefore, we may assume
		\begin{equation}\label{eq:edgeG_l}
			\big|E(G)\setminus E(G_\l)\big|
			\leq 
			\eps pn^2\,.
		\end{equation}
		Let $\phi\colon E(G)\to\NN$ 
		be an edge colouring, which is $(\l^{-5}/4,p)$-bounded on $U\subseteq V(G)$ of size at least $\nu n$.
		
		Let $U'_1\dcup\dots\dcup U'_\l=U$ be a balanced partition of $U$. After removing a few vertices of too high degree, i.e., vertices $u\in U'_i$ for which 
		$d_{G}(u,U'_j)>1.1\,p|U'_j|$ for some $j\neq i$, we arrive at a collection $\cU=(U_1,\dots,U_\l)$ of mutually disjoint sets of size $m$ such that 
		\[
		|U_1|=\dots=|U_\l|=m\geq \frac{|U|}{2\l}\geq\frac{\nu}{2\l}n
		\]
		and for every every vertex $u_i\in U_i$ and $j\in[\l]$ we have
		\begin{equation}\label{eq:degU}
			d_{G}(u_i,U_j)\leq 2 p|U_j|=2pm\,.
		\end{equation}
		In addition, the boundedness of $\phi$ and the choice of $\delta$ ensures for every colour $c\in \NN$
		\begin{equation}\label{eq:dbddU}
			d_c(u_i,U_j)
			\leq
			d_c(u_i,U)
			\leq
			\frac{1}{4\l^5} p |U|
			\leq 
			\frac{1}{2\l^4} p |U_j|
			=
			\delta pm\,.
		\end{equation}
		In view of~\eqref{eq:edgeG_l}, part~\ref{it:trans1-3} of Theorem~\ref{t:transference1} yields 
		\[
		\kappa_\l(G_{\l}[\cU])
		\geq
		(1-2\eps)p^{\binom{\l}{2}}m^\l
		\]
		and below we shall bound the number of non-rainbow copies  of $K_\l$ in~$G[\cU]\supseteq G_\l[\cU]$. 
		
		For that it will be useful to classify the non-rainbow copies according to where the repeated colour occurs. 
		We consider two cases depending on whether the two edges of the same colour share a vertex or form a matching.
		Hence, we consider the number $\kappa_\l^{\ee}(G[\cU],\phi)$ of copies of~$K_\l$ in $G[\cU]$ 
		containing edges $e\in E_G(U_1,U_2)$ and $e'\in E_G(U_1,U_3)$
		such that~$\phi(e)=\phi(e')$. Similarly, we define $\kappa_\l^{\ep}(G[\cU],\phi)$ to be those copies with the two 
		edges of the same colour being from $E_G(U_1,U_2)$ and $E_G(U_3,U_4)$. We will exploit the boundedness of $\phi$ to 
		deduce the following claim from Lemma~\ref{l:upper-rainbow}.
		\begin{claim}\label{c:bad-upper}
			We have $\kappa_\l^{\ep}(G[\cU],\phi)\leq 5\delta p^{\binom{\l}{2}}m^\l$ and $\kappa_\l^{\ee}(G[\cU],\phi)\leq 5\delta p^{\binom{\l}{2}}m^\l$.
		\end{claim}
		Applying the claim to cover all $\binom{\binom{\l}{2}}{2}$ possibilities where the two identically coloured edges may appear
		within the pairs of classes of $\cU$ yields at least 
		\[
		(1-2\eps)p^{\binom{\l}{2}}m^\l - \frac{\l^4}{8}\cdot 5\delta p^{\binom{\l}{2}}m^\l
		>
		\frac{1}{2}p^{\binom{\l}{2}}m^\l
		\]
		rainbow copies of $K_\l$ in $G_\l[\cU]$. Hence, for the proof of Lemma~\ref{l:rainbow} it only remains to verify Claim~\ref{c:bad-upper}.
	\end{proof}

	\begin{proof}[Proof of Claim~\ref{c:bad-upper}]
		We first bound $\kappa_\l^{\ep}(G[\cU],\phi)$. Note that if each colour occupies $\delta p m^2$ 
		edges of~$G[U_1,U_2]$, then the claimed upper bound follows easily from Lemma~\ref{l:upper-rainbow}\,\ref{it:matching-bound}. 
		We shall reduce the problem to this case.
		
		Fix a partition $C_1\dcup\dots\dcup C_r=\NN$ of the set of colours  such that 
		\[
		r\leq \frac{2}{\delta}
		\qqand
		\big|\phi^{-1}(C_\rho)\cap E_G(U_1,U_2)\big|
		\leq
		2\delta pm^2\ \text{for every $\rho\in[r]$.}
		\]
		Note that due to~\eqref{eq:degU} and~\eqref{eq:dbddU} such a partition of the colours can be found greedily, by
		adding colours to a class $C_\rho$ as long as the bound on the number of edges in the preimage holds. 
		
		For $\rho\in [r]$ let $S_\rho$ be the subgraph obtained from $G[\cU]$ 
		by restricting the edges from~$G[U_1,U_2]$ and $G[U_3,U_4]$ to
		$\phi^{-1}(C_\rho) \cap E_G(U_1,U_2)$ and $\phi^{-1}(C_\rho) \cap E_G(U_3,U_4)$, i.e., to those edges having a colour from $C_\rho$.  
		Note that every copy counted by $\kappa_\l^{\ep}(G[\cU],\phi)$ is contained in some $S_\rho$.
		
		Applying Lemma~\ref{l:upper-rainbow}\,\ref{it:matching-bound}, we obtain
		\[
		\kappa_\l(S_\rho)
		\leq  
		2 \delta pm^2 \cdot \big|E_{S_\rho}(U_3, U_4)\big| \cdot p^{\binom{\l}{2}-2}m^{\ell-4} + \eps p^{\binom{\l}{2}}n^\ell,
		\]
		and summing over all $\rho\in[r]$ and recalling~\eqref{eq:degU} yields the desired bound
		\[
		\kappa_\l^{\ep}(G[\cU],\phi)
		\leq 
		4 \delta p^{\binom{\l}{2}}m^\l + \eps p^{\binom{\l}{2}}n^\ell
		\overset{\eqref{eq:eps-bdd}}{\leq} 
		5 \delta p^{\binom{\l}{2}}m^\l\,.
		\]

		For the bound on $\kappa_\l^{\ee}(G[\cU],\phi)$, we will again partition the colours, to reduce it to $2/\delta$ applications 
		of Lemma~\ref{l:upper-rainbow}\,\ref{it:cherry-bound}. However, every vertex in $U_1$ will define its own partition. 
		For every vertex $u\in U_1$ we fix a partition $C^u_1\dcup\dots\dcup C^u_{r_u}=\NN$ such that for every $\rho\in[r_u]$ we have
		\[
		\sum_{c\in C^u_\rho} d_c(u,U_2) 
		\leq
		2\delta p m^2\,. 
		\]
		Again it follows from~\eqref{eq:degU} and~\eqref{eq:dbddU} that such a partition exists for some $r_u\leq 2/\delta$. For simplicity we allow empty 
		partition classes and, hence,  we  may assume that $r_u=r=\lfloor 2/\delta\rfloor$ for every vertex $u\in U_1$.
		
		For $\rho\in [r]$ this time we let $S_\rho$ be the subgraph obtained from $G[\cU]$ 
		by restricting the edges in $G[U_1,U_2]$ and $G[U_1,U_3]$ incident to a vertex $u\in U_1$ to those, which received a colour from~$C^u_\rho$, i.e., 
		\[
		E_{S_\rho}(U_1,U_2)=\bigcup_{u\in U_1}\big\{uv\in E(G)\colon v\in U_2\tand \phi(uv)\in C^u_\rho\big\}
		\]		
		and $E_{S_\rho}(U_1,U_3)$ is defined in an analogous way.
		This definition guarantees that every $K_\l$ in $G[\cU]$  
		containing a monochromatic $K_{1,2}$ with center in $U_1$ and leaves in $U_2$ and $U_3$ is contained in $S_\rho$ for some $\rho\in[r]$.
		
		In view of Lemma~\ref{l:upper-rainbow}\,\ref{it:cherry-bound}, we have
		\begin{align*}
			\kappa_\l(S_\rho)
			&\leq 
			\sum_{u\in U_1} 2\delta pm\cdot \sum_{c\in C^u_\rho} d_c(u,U_3) \cdot p^{\binom{\l}{2}-2} m^{\ell-3} + \eps p^{\binom{\l}{2}}n^\ell\\
			&=
			\big|E_{S_\rho}(U_1,U_3)\big|\cdot 2\delta p^{\binom{\l}{2}-1} m^{\ell-2} + \eps p^{\binom{\l}{2}}n^\ell\,.
		\end{align*}
		In view of~\eqref{eq:degU}, summing over all $\rho\in[r]$ yields the desired bound
		\[
		\kappa_\l^{\ee}(G[\cU],\phi)
		\leq 
		4 \delta p^{\binom{\l}{2}}m^\l + \eps p^{\binom{\l}{2}}n^\ell
		\overset{\eqref{eq:eps-bdd}}{\leq}
		5 \delta p^{\binom{\l}{2}}m^\l\,,
		\]
		which concludes the proof of Claim~\ref{c:bad-upper} and, hence, Lemma~\ref{l:rainbow} is established.
	\end{proof}
	
	\begin{rem}
		Theorem~\ref{t:transference1} from~\cite{cgss14} is more general and applies not only to 
		cliques $K_\l$, but to all \emph{strictly balanced} graphs~$H$ for 
		$p\geq Cn^{-\frac{|V(H)|-2}{|E(H)|-1}}$, i.e., for graphs $H$ satisfying 
		\[
		\frac{|E(H')|-1}{|V(H')|-2}
		<
		\frac{|E(H)|-1}{|V(H)|-2}
		\] 
		for all proper subgraphs $H'\subsetneq H$ on at least three vertices. 
		Starting with this general version of Theorem~\ref{t:transference1}, the arguments 
		from this section can be carried out 
		verbatim for such graphs~$H$. This yields a version of Lemma~\ref{l:rainbow} guaranteeing 
		rainbow copies of strictly balanced graphs~$H$ for $(\delta,p)$-bounded edge colourings 
		of~$G(n,p)$ for $p\geq Cn^{-\frac{|V(H)|-2}{|E(H)|-1}}$ as long as $\delta$ is sufficiently 
		small and $C$ is sufficiently large depending on $H$.
	\end{rem}
	
	\section{Canonical cliques in unbounded colourings of random graphs}	
	\label{sec:strictly-min}
	This section contains the proof of Lemma~\ref{l:strictly-min} along with the required prerequisites.
	We begin with an overview of the proof.
	\subsection{Outline of the proof}
	\label{sec:outline-unbounded}
	Again, the main tool in the proof of Lemma~\ref{l:strictly-min} is the transference principle developed by Conlon and Gowers. 
	This result asserts that asymptotically almost surely every subgraph $F$ 
	of $\Gnp$ has a dense model (see, e.g., Theorem~\ref{t:transference1}\,\ref{it:trans1-1}). 
	Moreover, if $p \gg n^{-\frac {2}{\l+1}}$, then the dense model $D$ and the subgraph~$F$, 
	have closely related distributions {of} the copies of $K_\l$, which is made precise 
	in part~\ref{it:trans1-2} of Theorem~\ref{t:transference1}. Roughly speaking, 
	we may think of $F$ being \emph{close to} a random subgraph of $D$ whose edges are sampled independently with probability $p$.
	
	For mimicking the argument from Section~\ref{sec:outline} for colourings of $K_n$, the main obstacle is that the neighbourhoods 
	in $G(n,p)$ are of order $pn$, not allowing the iteration rendered in~\eqref{eq:iteration-dense}. 
	We circumvent this by considering \textit{suitable} subgraphs $F_1, F_2, \ldots, F_{L}$ of $G(n,p)$ for~$L=2(\l-1)(\l-2)+2$ 
	and obtain their dense models $D_1,\ldots, D_{L}$, which yield linear-sized sets as neighbourhoods in the dense models. 
	Another challenge is that, for transference to be useful, the dense model has to contain~$\Omega(n^{\ell})$ copies of $K_\ell$, 
	so we also need a robust, counting version of the argument for the case of unbounded colourings of~$K_n$.
	
	In the proof of Lemma~\ref{l:strictly-min} the ordering of the underlying vertex set will be important. For that, we refine
	the definition of $d_c(u,U)$ for a given edge colouring $\phi$ and for $\diamond\in\{<\,,>\}$ we set
	\[
	d^\diamond_c(u,U)=\big|N^\diamond_c(u,U)\big|=\big|\{w\in U\colon \phi(uw)=c\tand u\diamond w\}\big|\,.
	\]
	In Lemma~\ref{l:strictly-min} we assume that every sufficiently large set~$U$ of vertices of the random graph has the property that 
	half of its vertices $u$ have a large monochromatic neighbourhood. For simplicity, in the outline below we shall always 
	assume that, in fact, these vertices~$u$ always lie before their monochromatic neighbourhood, 
	i.e., $d^<_c(u,U)\geq4\delta p |U|$ for at least~$|U|/4$ vertices.

	The subgraphs $F_i\subseteq G(n,p)$ for $i\in[L]$ will be selected in an iterative manner. For the definition of $F_1$, 
	let $V_1$ be those $n/4$ vertices $v$ such that $d^<_{c(v)}(v)\geq 4\delta pn$ for some associated colour $c(v)$.
	The graph $F_1$ is then defined as the union of those edges, i.e.,
	\[
	E(F_1)=\{vw\in E(G(n,p))\colon v\in V_1, v<w,\tand\phi(vw)=c(v)\}\,.
	\]
	In other words, the colour of each edge $e\in E(F_1)$ is defined by its \textit{starting point} $\min(e)$. 
	The transference principle yields a dense model $D_1$ of $F_1$, and by part~\ref{it:trans1-1deg} of Theorem~\ref{t:transference1}, for
	most $v \in V_1$, $p\cdot d^<_{D_1}(v)$ is well approximated by $d^<_{F_1}(v)$. Consequently, typically
	the neighbourhood of $v$ in $D_1$ defines a subset $S_1(v)$ of size at least $3\delta n$.
	
	We continue and define $F_2$ based on the sets $S_1(v_1)$ for $v_1\in V_1$. Again, appealing to the 
	assumption~\eqref{eq:delta-unbounded-undir} of Lemma~\ref{l:strictly-min} as described above for $U=S_1(v_1)$, 
	we obtain $|S_1(v_1)|/4$ pairs $(v_2, c)\in S_1(v_1)\times \NN$ with
	\begin{equation}\label{eq:outline-deg2}
		d_{c}^<(v_2, S_1(v_1)) \geq 4 \delta p |S_1(v_1)|=\Omega(\delta^2 p n)\,.
	\end{equation}
	However, the colour $c$ depends on $v_1$ and, similarly, as in the definition of~$F_1$ we shall find a ``large'' 
	monochromatic neighbourhood of~$v_2$ for the definition of~$F_2$. On the other hand, since the degree of 
	$v_2$ is close to $pn$ in $G(n,p)$, at most $1/\delta^{2}$ different colours might occur for different 
	choices of~$v_1$. We let $c(v_2)$ be the \textit{majority} colour and restrict to it for the definition of $F_2$.
	Moreover, for an appropriately chosen subset $V_2\subseteq  \bigcup\{S_1(v_1)\colon v_1\in V_1\}$, 
	we shall define $F_2$ in such a way that
	\[
	E(F_2)\subseteq \{v_2w \colon v_2 \in V_2, v_2<w,\tand\varphi(v_2w) = c(v_2) \}\,.
	\]
	In particular, the colours of the edges in $F_2$ are determined again by their starting point. For some technical reasons the definition of $F_2$ is a bit more involved and, for example, we will impose $F_2$ to be bipartite (see Claim~\ref{c:hypergraphs-stronger} below for the formal statement).
	Similar as before, we obtain a dense model~$D_2$ for $F_2$. Owing to~\eqref{eq:outline-deg2},
	the vertices in $V_2$ have $\Omega(\delta^2pn)$ neighbours in $F_2$. Consequently, most vertices in $V_2$ have $\Omega(\delta^2n)$ neighbours 
	in $D_2$. Those neighbourhoods yield the linear sized sets $S_2(v_2)$, which allows us to iterate the argument to obtain a subgraph $F_3 \subseteq G(n,p)$ with all edges colours determined by its starting vertex and a dense model $D_3$ of $F_3$ with linear sized neighbourhoods.
	
	This way we obtain graphs $F=F_1\cup\dots\cup F_{L}$, and $D=D_1\cup\dots\cup D_{L}$. We shall show that the dense graph 
	$D$ contains $\Omega(n^{\l})$ copies of $K_\l$ by its construction through ``nested neighbourhoods.'' Consequently, 
	the transference principle in the form of part~\ref{it:trans1-2} of Theorem~\ref{t:transference1}
	tells us that the sparse subgraph $F\subseteq G(n,p)$ 
	contains $\Omega(p^{\binom{\l}{2}}n^\l)$ copies of~$K_\l$ and by the choice of the $F_i$ these copies will be 
	non-strictly min-coloured.

	It remains to analyse the induced vertex colourings of these non-strictly min-coloured copies of $K_\ell$ in $G(n,p)$.
	However, those induced vertex colours are not synchronised within each of the sets $V_1,\dots, V_L$, i.e., the edges of $F_i$
	are not monochromatic and may induce different colours for the vertices in $V_i$. 
	We address this issue by a more careful selection of the graphs $F_1,\dots,F_L$
	(see~\ref{it:psi} of Claim~\ref{c:hypergraphs-stronger} below). More specifically, we distinguish between two cases. Firstly, if most of the sets $V_i$ have a `dominant vertex colour' $c_i$, then one can find a monochromatic or a rainbow subset of the vertex colours $c_i$ of size $\ell-1$. Any clique in the corresponding $\ell$-partite subgraph of $F$ will be monochromatic or min-coloured. Secondly, assume that there are vertex sets $V_1, \ldots, V_{\ell-1}$ (after re-indexing) have no `dominant' vertex colour. Then a counting argument will be applied to show that a small proportion of $K_\l$-copies in $F_1, \ldots, F_{\ell-1}$ contain distinct vertices $u, v$ with $c(u)=c(v)$, and therefore most of the $K_\l$-copies found above will be strictly min-coloured.

	The proof of Lemma~\ref{l:strictly-min} is based on a more involved application of the transference principle compared to the proof of Lemma~\ref{l:rainbow} in Section~\ref{sec:rainbow}. In Section~\ref{sec:prelim} we review the required results from~\cite{cg16}
	and some helpful consequences for the proof of Lemma~\ref{l:strictly-min} , which is presented in Section~\ref{sec:pf-unbounded}.

	\subsection{The transference principle}
	\label{sec:prelim}
	In this section, we present the prerequisites for the proof of Lemma~\ref{l:strictly-min}. The statements are written for arbitrary strictly balanced graphs $H$, although we shall only employ them for $H=K_\ell$. For a strictly balanced graph $H$ on at least three vertices, 
	we can define its \emph{$2$-density} by
	\[
	m_2(H) = \frac{|E(H)|-1}{|V(H)|-2}\,,
	\]
	and we note that $m_2(K_\ell) = (\ell+1)/2$ appears in the exponents of~$p=p(n)$ in the assumptions of the earlier 
	statements in Sections~\ref{sec:introduction}--\ref{sec:rainbow}.
	Similarly, below we shall impose that, for a given
	strictly balanced graph $H$ on the vertex set $[\ell]$ (with $\ell \geq 3$) and a sufficiently large constant $C$, $p=p(n)\geq Cn^{-1/m_2(H)}$.	
	
	For some (large) integer $n$, we shall work within the set of functions from $[n]^{(2)}$ 
	to $\RR$, which naturally corresponds to the set of \emph{weighted graphs} on the vertex set $[n]$.
	Therefore, we often identify a graph $F$ on $[n]$ with the indicator function $\ind_F$ of its edge set and a
	dense model $\df$ for $F$ is a function $\df\colon [n]^{(2)} \to [0, 1]$, 
	which is ``close'' to $p^{-1}\ind_F$ 
	in terms of its distribution {of} weighted edges and copies of $H$.
	The definitions of edge counts and vertex degrees extend straightforwardly to functions 
	$\ff\colon [n]^{(2)} \to \R$ and for that we set
	\[
	e(\ff) = \sum_{u < w} \ff(uw)\,,\quad
	e_\ff(U, W) = \sum_{u \in U, w \in W} \ff(uw)\,,\qand
	d_\ff(v, U) = \sum_{u \in U \setminus \{ v\}}\ff(uv).
	\]
	With this notation at hand we can define the \emph{cut-norm} by
	\[
	\cutnorm{\ff} 
	=
	\frac{1}{n^2} \max_{U, W \subseteq [n]} |e_{\ff}(U, W)\,
	|.\]
	This norm allows us to compare the edge distributions of two weighted graphs on $[n]$. In fact, we will consider
	weighted graphs $\ff$ and $\df$ to be ``close,'' if $\cutnorm{\ff-\df}$ is ``small''  (see, e.g., 
	Theorem~\ref{t:transference1}\,\ref{it:trans1-1} and Theorem~\ref{t:transference}\,\ref{it:transf-edgedist} below).
	Similarly, for a graph $H$ with vertex set $[\l]$ we define its homomorphism density in $\ff$ by 
	\[
	\Lambda_H(\ff) = \frac{1}{n^\ell}\sum_{v_1, \dots, v_\ell\in[n]}\prod_{ij \in e(H)}\ff(v_iv_j)\,,
	\]
	where we use the convention $\ff(v_iv_j)=0$ in case $v_i=v_j$. Consequently, for cliques the quantity 
	$\Lambda_{K_\l}(\ff)$ corresponds to the weighted $K_\l$-density in $\ff$ and for $\ff$ being an unweighted  
	graph we can recover the notation $\kappa_\ell (\ff)=n^\l\cdot\Lambda_{K_\l}(\ff)$ from Section~\ref{sec:rainbow}.

	As discussed in the outline for the proof of  Lemma~\ref{l:strictly-min},  
	Theorem~\ref{t:transference1} will be applied in stages to the 
	subgraphs $F_1,\dots,F_L \subseteq \Gnp$ to obtain dense models $D_1, \dots, D_L$. 
	In order to ensure that $D=D_1 \cup \dots \cup D_L$ is still a useful approximation 
	of $F=F_1 \cup \dots \cup F_L$, we need a bit more insight into how these dense models are obtained.  
	Informally, Conlon and Gowers~\cite{cg16} construct a norm $\| \cdot \|$ on  the set of weighted graphs 
	$\R^{[n]^{(2)}}$ so that the following holds:
	If $\ind_F$ is the characteristic function of the edges of some 
	$F\subseteq \Gnp$ and $\|  p^{-1} \ind_F - \df \|$ is sufficiently small for some dense model~$\df$ with 
	$\| \df \|_{\infty} \leq 1$, then $p^{-1}\ind_F$ and~$\df$ have a ``similar'' 
	distribution of edges and copies of~$H$.  
	A major contribution of~\cite{cg16} is precisely finding a norm which is sufficiently weak to allow a dense model which is arbitrarily close to~$p^{-1}\ind_F$, and sufficiently strong to preserve the relevant properties of~$F$. However, the norm $\|\cdot \|$ actually depends on the random graph $\Gnp$
	in the sense that asymptotically almost surely $G\in G(n,p)$ has the property that there is a norm $\|\cdot \|$ with the aforementioned 
	properties for every subgraph $F\subseteq G$. Theorem~\ref{t:transference} below is a version of the transference principle of Conlon and Gowers, which is tailored for our proof of Lemma~\ref{l:strictly-min}. It is implicit
	in the work in~\cite{cg16} and in the Appendix we discuss in more detail how it can be extracted.

	\begin{thm} \label{t:transference} 
		For every strictly balanced graph $H$ with $V(H)=[\l]$ and every $\eps>0$ there is some constant $C>0$ such that for $p=p(n)$ with 
		$Cn^{-1/m_2(H)}\leq p\leq 1/C$ asymptotically almost surely the following holds for $G\in G(n,p)$.
		
		There exists a norm $\|\cdot \|$ on the set of weighted graphs $\R^{[n]^{(2)}}$ such that
		for every~$F \subseteq G$, there is a dense model $\df_F\colon [n]^{(2)}\to[0,1]$ with $\|p^{-1}\ind_F  - (1+ \eps)\df_F  \| < \eps$
		and 
		\begin{enumerate}[label=\rmlabel]
			\item 	\label{it:transf-edgedist}
			for all functions $\ff$, $\df \colon [n]^{(2)}\to\R$ with $\| \ff   - (1+\eps)\df \| < \eps$ we have
			\[
			\cutnorm{\ff - \df } \leq 2\eps\,.
			\]
			\item \label{it:transf-Hcount}
			for every function $\df\colon [n]^{(2)}\to[0,2\l]$ with $\| p^{-1}\ind_F  - (1+\eps) \df \| \leq \eps$ we have 
			\[
			\Lambda_H(\df ) \leq  p^{-|E(H)|}\Lambda_H(F) + (4\ell)^{\ell^2}\cdot\eps \,.
			\]
		\end{enumerate}
	\end{thm}
	We emphasise that, while~\ref{it:transf-edgedist} applies to any weighted graph $\ff$, 
	part~\ref{it:transf-Hcount} only applies to the subgraphs~$F$ of the random graph.
	Anyhow, in our intended application we have~$\ff=p^{-1}\ind_F$. Theorem~\ref{t:transference} is closely related 
	to Theorem~\ref{t:transference1} where $S$ and $D$ in Theorem~\ref{t:transference1} take the r\^ole of $F$ and $\df_F$ 
	in Theorem~\ref{t:transference}. In fact, applying~\ref{it:transf-edgedist} and~\ref{it:transf-Hcount}
	with~$\df=\df_F$ and $\ff=p^{-1}\ind_F$ in Theorem~\ref{t:transference} implies statements~\ref{it:trans1-1}
	and~\ref{it:trans1-2} of Theorem~\ref{t:transference1}.
	
	It will be convenient to move from the dense weighted graphs $\df$ back to unweighted graphs $D$ sampled by the edge weights of~$\df$. 
	The following lemma follows directly from  Chernoff's inequality and a union bound.
	\begin{lemma} \label{l:edge-sampling}
		For every $\eps>0$ and any sequence of functions $\df\colon [n]^{(2)} \to [0,1]$ asymptotically almost surely we have 
		$\cutnorm{\ind_D - \df} \leq \eps$ for the random graph $D$ on $[n]$ with every edge $e$ appearing independently 
		with probability $\df(e)$. \qed
	\end{lemma}
	
	We will also use the following \emph{counting lemma} comparing the number of subgraphs of two graphs in 
	terms of the cut-norm. This can be viewed as the global counting lemma for the weak regularity lemma 
	of Frieze and Kannan~\cite{fk99} and it can be found in~\cite{ls06}*{Lemma~4.1}.
	\begin{lemma} \label{l:cutnorm-copies}
		\pushQED{\qed}
		For every graph $H$ and all functions $\ff$, $\df\colon [n]^{(2)}\to [0,1]$ we have
		\[
		\big|\Lambda_H(\ff) - \Lambda_H(\df)\big| 
		\leq 
		2e(H)\cdot\cutnorm{\ff - \df}\qedhere
		\]
	\end{lemma}
	
	We conclude this section with the following fact that the cut-norm controls most vertex degrees into given subsets 
	(see, e.g., the deduction of~\ref{it:trans1-1deg} of Theorem~\ref{t:transference1} from~\ref{it:trans1-1}).
	\begin{lemma} \label{l:degree-cutnorm}
		For every $\eps>0$ and all functions $\ff$, $\df\colon [n]^{(2)}\to [0,1]$ with $\cutnorm{\ff - \df} \leq \eps$
		the following holds. 
		For all $U\subseteq [n]$ with  $|U| > 2 \eps^{1/3}n$, all but at most $\eps^{1/3}n$ vertices $v \in [n]$ 
		satisfy
		\[
		|d_\ff(v, U) - d_\df(v, U)| \leq \eps^{1/3}|U|\,.
		\]
	\end{lemma}
	\begin{proof}[Proof of Lemma~\ref{l:degree-cutnorm}]
		Let $S$ be the set of vertices $v$ with $d_\ff(v, U) - d_\df(v, U) \geq \eps^{1/3}|U|$.
		We have
		\begin{align*}
			n^2 \cutnorm{\ff - \df} 
			\geq 
			\sum_{v \in S} d_\ff(v, U)- d_\df(v, U)
			\geq  
			\eps^{1/3} |U||S|
			\geq 
			2\eps^{2/3} n|S|
		\end{align*}
		and $|S|\leq \eps^{1/3}n/2$ folllows. Similarly, one can show that there are at most $\eps^{1/3}n/2$ vertices~$v$ such that  $d_\df(v, U) - d_\ff(v, U) \geq \eps^{1/3}|U|$ and the claimed bound follows.
	\end{proof}

	\subsection{Proof of Lemma~\ref{l:strictly-min}}	
	\label{sec:pf-unbounded}
	Given $\l\geq 3$ and $\delta>0$, we fix auxiliary constants 
	\[
	L=2(\l-1)(\l-2)+2\,,\quad
	\nu = \frac{\delta^{4\ell^4 +2\ell^2}}{10^{4\ell^2}}\,,\quad
	\alpha = \frac{\nu}{8\ell^2}
	\,,\qand
	\zeta=\delta^{10\l^2+2}\,.
	\]
	Moreover we fix auxiliary constants $\eps_1$ and $\eps_2$ and the desired $C$ to satisfy the hierarchy\footnote{Here, $x \gg y$ means that for a given $x$, $y$ is taken sufficiently small so that all the following claims hold.}
	\[
	\delta\,,\l^{-1} \gg \nu, \zeta \gg \alpha \gg \eps_2\gg \eps_1 \gg C^{-1}\,.
	\] 
	Assume that $G \in G(n,p)$ satisfies the conclusion of Theorem~\ref{t:transference} and Theorem~\ref{t:transference1}~\ref{it:trans1-3}  
	for $\eps = \eps_1$. Let $\norm{\,\cdot\,}$ denote the norm given by Theorem~\ref{t:transference}. 
	Moreover, we may assume that all vertices in $G$ have degree $(1 \pm \eps_1)pn$, since all these properties hold asymptotically almost surely.
	
	For the moreover-part of the lemma, 
	we have in addition  $p=p(n)\leq n^{-\frac{2\l-5}{\l^2-\l-4}}/\omega(n)$ for some function 
	$\omega$ tending to infinity. Similarly as in the proof of Lemma~\ref{l:rainbow} this allows us to assume 
	\begin{equation}\label{eq:edgeG_l2}
		\big|E(G)\setminus E(G_\l)\big|
		\leq 
		\zeta pn^2\,.
	\end{equation}
	By our choice of $\zeta$, this implies that the crucial assumption~\eqref{eq:delta-unbounded-undir} for~$G$ in 
	Lemma~\ref{l:strictly-min} extends to~$G_\l$ at the price of a small change of the constants. Namely, every $U\subseteq V(G_\l)$ with size $|U|\geq \delta^{5 \l^2} n$
		satisfies
		\[
			\big|\{u \in U\colon d_{G_\l,c}(u, U) \geq 7.9\delta p |U| \text{ for some colour } c\} \big|
			\geq 
			0.49|U|\,,
		\]
    and the proof given below 
	can be carried out in $G_{\l}$ as well.

	Hence, let us return to the proof in the original graph $G$. For a comparability sign $\diamond \in \{<,\,>\}$ set	
	$$	
	B^\diamond(U) = \{v \in U\colon  d_c^\diamond(v, U) \geq 4\delta p |U| \text{ for some colour } c\}\,.
	$$
	The assumption~\eqref{eq:delta-unbounded-undir} implies that  for every set $U$ with $|U|\geq \delta^{5 \ell^2} n$, we have
	\begin{equation} \label{eq:delta-unbounded}
		|B^<(U)| \geq \frac{|U|}{4}\quad \text { \ or \ }\quad |B^>(U)| \geq \frac{|U|}{4}.
	\end{equation}
	
	The condition~\eqref{eq:delta-unbounded} will be iterated to inductively build some structures in $G$, 
	as detailed in the following claim.  We will build subgraphs $F_t$ of $G$ for $t\in[L]$,
	which are non-strictly min-coloured or max-coloured and \textit{relatively dense} to $G$. Moreover, we consider the dense models 
	$D_t$ of those $F_t$ given by Theorem~\ref{t:transference}. The hypergraph $\Hc$ is used to keep track of cliques in $\bigcup_t D_t$. In the statement and proof, we usually identify a hypergraph $\Hc$ with its set of edges. In particular,~$|\Hc|$ is the number of edges in $\Hc$.
	
	\begin{claim} \label{c:hypergraphs-stronger}
		For every $t\in[L]$ there is a set of vertices $V_t$ disjoint from $V_1\cup \dots \cup V_{t-1}$ with $|V_t|\geq \frac{1}{25}\alpha \delta^{2t}n$,
		a comparability sign $\diamond_t \in \{<,\,>\}$, and a colour index $\psi_t \in \N \dcup \{ \star\}$ 
		such that the following holds:
		\begin{enumerate}[label=\alabel]
			\item \label{it:Fi} Each $v \in V_t$ is assigned a colour $c(v)$ and there is a graph $F_t\subseteq G$ 
			whose edges $vu$ satisfy $v \in V_t$ and $u \in N^{\diamond_t}_{c(v)}(v) \setminus (V_1 \cup \dots \cup V_t)$.
			\item \label{it:psi} If $\psi_t\in\NN$, then $c(v) = \psi_t$ for all $v \in V_t$. Otherwise, if $\psi_t=\star$, 
			then we have $|\{ v \in V_t\colon  c(v) = j \}| \leq \alpha |V_t|$ for every colour $j\in\NN$.
			\item \label{it:Di} There is a function $\df_t\colon [n]^{(2)} \to [0,1]$ with
			$$\|p^{-1}\ind_{F_t} - (1+\eps_1)\df_t \| \leq \eps_1\,,$$
			and a graph $D_t$ with
			$$\cutnorm{D_t - \df_t}\leq 7 \eps_1 \text{\quad and \quad} \cutnorm{D_t - p^{-1}\ind_{F_t}}\leq 9\eps_1\,.$$
			\item \label{it:Sv} 
			There is a $t$-partite $t$-uniform hypergraph $\Hc_t$ on $V_1 \cup \dots \cup V_t$ with
			$$|\Hc_t| \geq 10^{-2t}\delta^{t^2} |V_1| \cdots |V_t|$$
			and for any $(v_1, \ldots, v_t) \in \Hc_t$, there is a set $S(v_1, \ldots, v_t)$ disjoint from $V_1, \ldots, V_t$ of size at least $\delta^t n$ such that for $i \leq t$,
			$$N_{D_i}(v_i) \supseteq \{v_{i+1}, \ldots, v_t \}\cup S(v_1, \ldots, v_t)\,.$$
		\end{enumerate}
	\end{claim}
	
	\begin{proof}[Proof of Claim~\ref{c:hypergraphs-stronger}]
		We start the induction with $t=1$. 
		Let $B = B^{\diamond_1}(V(G))$ be the set of~$n/4$ vertices given 
		by~\eqref{eq:delta-unbounded}. For $v \in B$, let $c(v)$ be a colour in which 
		$$d^{\diamond_1}_{c(v)}(v) \geq 4 \delta pn .$$ 
		
		Moreover, we wish to transfer a \textit{bipartite} graph $F_1$. To this end,  let each vertex of $B$ be placed into a set $Z_1$ independently at random with probability $1/2$, and let $Y_1 = V(G) \setminus Z_1$.  We may assume that $|Z_1| \geq n/10$, and each $v \in Z_1$ satisfies
		\begin{equation} \label{eq:deg-basis}
			d^{\diamond_1}_{c(v)}(v, Y_1) \geq 1.5 \delta pn
		\end{equation}
		as this happens asymptotically almost surely.
		
		If there is a colour $j$ such that $|c^{-1}(j)\cap Z_1| \geq \alpha |Z_1|$, then we set $V_1 = c^{-1}(j)$ and $\psi_1 = j$. Otherwise take $V_1 = Z_1$ and set $\psi_1 = \star$. Note that in either case, 
		\[
			|V_1| 
			\geq 
			\frac{\alpha n}{10}\,.
		\]
		Let $F_1$ be the subgraph of $G$ with
		\[
			E(F_1) = \{vu\colon  v \in V_1\,,\ u \in Y_1\,,\ v \diamond_1 u\,, \tand \varphi(v u ) = c(v)\}\,.
		\]		
		Let $\ind_{F_1}$ be the characteristic function of $F_1$. By Theorem~\ref{t:transference}, there is a  weighted graph 
		$\df _1\colon [n]^{(2)}\to[0, 1]$ with $\|p^{-1}\ind_{F_1} - (1+\eps_1)\df _1 \| \leq \eps_1$. In particular, 
		$\cutnorm{p^{-1}\ind_{F_1} - \df _1}\leq 2 \eps_1$ by Theorem~\ref{t:transference}~\ref{it:transf-edgedist}.
		
		Let $\df_1 '= \df_1 \! \restriction_{V_1 \times ([n] \setminus V_1)}$. Passing to a restriction of $\df_1$ is just a technicality to circumvent small overlaps between $\df_1, \ldots, \df_t$, and we will now show that $\cutnorm{\df_1 - \df_1'} \leq 6 \eps_1$. Denote $\ff_1 = p^{-1} \ind_{F_1}$.  To bound $\df_1 - \df_1'$, notice that, by definition of $\df_1'$ and since all edges of $F_1$ lie in $V_1 \times ([n] \setminus V_1)$,
		$$\frac{1}{n^2} \left| e_{\df_1'}([n])- e_{\ff_1}([n])\right|=
		\frac{2}{n^2}  \left| e_{\df_1}(V_1, [n] \setminus V_1) - e_{\ff_1}(V_1, [n] \setminus V_1)\right| \leq 2\cutnorm{\df_1 - \ff_1} \leq 4 \eps_1\,.
        $$
		Hence $e_{\df_1}([n])-e_{\df_1'}([n]) = e_{\df_1}([n])-e_{\ff_1}([n])  + e_{\ff_1}([n])-e_{\df_1'}([n]) \leq 6\eps_1 n^2$. Since $\df_1-\df_1' \geq 0 $ (pointwise), we have
		$$\cutnorm{\df_1 - \df_1'} = \frac{1}{n^2}(e_{\df_1}([n])-e_{\df_1'}([n])) \leq 6\eps_1.$$
		Moreover,
		$$\cutnorm{\df_1' - p^{-1}\ind_{F_1}} \leq 8\eps_1$$
		using the triangle inequality.
		
		Let $D_1$ be a graph sampled from $\df' _1$. Asymptotically almost surely, $\cutnorm{\ind_{D_1} - \df_1'} \leq \eps_1$ by Lemma~\ref{l:edge-sampling}. Therefore, using the triangle inequality, we may assume that
		$$\cutnorm{\ind_{D_1} - \df_1}\leq 7 \eps_1 \text{\quad and \quad} \cutnorm{\ind_{D_1} - p^{-1}\ind_{F_1}}\leq 9\eps_1\,.$$
		Let $\Hc_1$ be the set of vertices $v \in V_1$ with
		$$d_{D_1}(v, Y_1) \geq   \delta n\,.$$
		By Lemma~\ref{l:degree-cutnorm}, using~\eqref{eq:deg-basis}, $\cutnorm{ \ind_{D_1}-p^{-1}\ind_{F_1}} \leq 9\eps_1$, and taking  $\eps_1< \eps_2^4$, we have that $|\Hc_1| \geq |V_1| - \eps_2 n \geq \frac{|V_1|}{2}$. This completes the case $t=1$.

		Suppose the Claim holds for $1, 2, \ldots, t-1$. Let $X =  [n]\setminus (V_1 \cup \dots \cup V_{t-1})$. For $(v_1, \ldots, v_{t-1}) \in \Hc_{t-1}$, consider the set $S(v_1, \ldots, v_{t-1}) \subseteq X$ as stated in~\ref{it:Sv}. Specifically, $|S(v_1, \ldots, v_{t-1})| \geq \delta^{t-1} n$. Denote $\xi = \xi(t) := \delta^{t-1}$. Applying the assumption~\eqref{eq:delta-unbounded}, we obtain a set $B^\diamond(S(v_1, \ldots, v_{t-1}))$ of order at least $\frac{\xi n}{10}$ for some $\diamond =\diamond (v_1, \ldots v_{t-1})$ such that for every $v \in B^\diamond(S(v_1, \ldots, v_{t-1}))$ and some colour $c = c(v_1, \ldots, v_{t-1}, v)$,
		\begin{equation} \label{eq:deg-v-1}
			d^{\diamond}_{c}(v, S(v_1, \ldots, v_{t-1})) \geq 4 \delta \xi pn\,.
		\end{equation}
		
		The next step is to remove the dependency of $\diamond$ and $c$ on $(v_1, \ldots, v_{t-1})$. Firstly, let $\Hc'_{t-1}$ be a subhypergraph of $\Hc_{t-1}$ of order at least $\frac 12 |\Hc_{t-1}|$ such that $\diamond(v_1, \ldots v_{t-1}) = \diamond_{t}$ for all $(v_1, \ldots, v_{t-1}) \in \Hc'_{t-1}$.
		
		Now form an auxiliary bipartite graph $\Jn$ with parts $ \Hc'_{t-1}$ and $X \times \N $ as follows: an edge $((v_1, v_2, \ldots v_{t-1}), (v, c ))$ in $\Jn$ means that $v \in S(v_1, \ldots, v_{t-1})$ and
		\begin{equation}	~\label{eq:deg-v}
			d_c^{\diamond_{t}}(v, S(v_1, \ldots, v_{t-1})) \geq 4 \delta \xi p n\,.
		\end{equation}
		Since every $(v_1, \ldots v_{t-1}) \in \Hc'_{t-1}$ is contained in at least $\xi n /10$ edges in $\Jn$ (one for each element of $B^{\diamond_t}(v_1, \ldots, v_{t-1})$), we have $$|\Jn|
		\geq \frac 12 |\Hc_{t-1}|\cdot \frac{\xi n}{10}\,.$$
		
		Let $X' \subset X$ be the set of vertices $v$ such that some $(v, c)$ is incident to an edge of $\Jn$, and note that for each $v$, there are at most $(3\delta \xi)^{-1}$ such colours $c$ -- this follows from $d_c(v) \geq 4\delta \xi pn$ and $d_G(v) \leq (1+\eps_1)pn$. 
		For each $v \in X'$, let $c(v)$ be the colour which maximises the degree of $(v, c)$ in  $\Jn$. Form $\Jn'$ from $\Jn$ by deleting all the vertices $(v, c^\dagger) $ with $c^\dagger \neq c(v)$; we have
		$$|\Jn'| \geq |\Jn| \cdot 3 \delta \xi \geq \frac {3}{20} |\Hc_{t-1}| \delta \xi^2 n\,.$$
		Since now each vertex $v \in X'$ is associated with a unique colour  $c(v)$, we may assume that one vertex part of $\Jn'$ is just $X'$.
		
		We wish the graph $F_t$ to be bipartite, so 
		let us split the set $X'$ as follows. Let $W_t \subseteq X'$ consist of vertices sampled from $X'$ independently at random with 
		probability $1/2$. Let $Y_t = X \setminus W_t$. With positive probability, 
		$$|\Jn'[E(\Hc_{t-1}'), W_t]|\geq \frac{3}{50}|\Hc_{t-1}| \delta \xi^2 n\,,$$
		and for each $((v_1, \ldots, v_{t-1}), v) \in \Jn'$ (recalling~\eqref{eq:deg-v}),
		\begin{equation} \label{eq:v-deg}
			d^{\diamond_{t}}_{c(v)}(v, Y_{t} \cap S(v_1, \ldots, v_{t-1})) \geq 1.5 \delta \xi pn\,,  
		\end{equation}
		where we used Chernoff bounds and the union bound. Thus we may assume that these two inequalities are satisfied.
		
		Now, let $Z_t \subseteq W_t$ be the set of vertices of degree at least $ \frac{1}{50}|\Hc_{t-1}| \delta \xi^2 $ in $J'$, and let $\Jn^*$ be the induced subgraph of $\Jn'$ on $(E(\Hc_{t-1}'), Z_t)$. 
		Since the vertices in $W_t \setminus Z_t$ were incident to at most $ \frac{1}{50}|\Hc_{t-1}| \delta \xi^2 n $ edges in total, we have
		$$|\Jn^*|\geq \frac{1}{25}|\Hc_{t-1}| \delta \xi^2 n\,.$$
		Recalling that $\Jn^*$ is a bipartite graph on the vertex sets $(E(\Hc_{t-1}'), Z_t)$, we have the lower bound
		\begin{equation}    \label{eq:Wt-size}
			|Z_t| 
			\geq \frac{|\Jn^*|}{|\Hc_{t-1}|} 
			\geq \frac{1}{25}  \delta \xi^2 n 
			= \frac{1}{25} \delta^{1+2(t-1)} n  
			\geq \frac{1}{25}\delta^{2t}n\,.
		\end{equation}

		Now, to ensure~\ref{it:psi}, if there is a colour $j$ such that
		$$Z_{t,j}=\{v \in Z_t\colon c(v) = j\}$$
		contains at least $\alpha |Z_t|$ vertices, let $V_t = Z_{t,j}$ and set $\psi_t = j$ (recalling that $\alpha \ll \delta, \ell^{-1}$ is a constant).
		Otherwise, set $V_t = Z_t$ and $\psi_t = \star$. 
		Using the minimum degree of the vertices from $V_t$, we have
		$$\big|\Jn^*[E(\Hc_{t-1}'), V_t]\big|
		\geq 
		\frac{1}{50}|\Hc_{t-1}| \delta \xi^2 |V_t|\,.$$
		Moreover, $|V_t| \geq \alpha |Z_t| \geq \frac{1}{25} \alpha \delta^{2t}n$, as required by the claim.

		Let $F_t$ be the subgraph of $G$ with
		$$E(F_t) = \{vy\colon  v \in V_t, y \in Y_t, v \diamond_t y , \tand \varphi(v y ) = c(v)\}\,.$$
		Recall that if $((v_1, \ldots, v_{t-1}), v_{t}) \in \Jn^{*}$, then by~\eqref{eq:v-deg}, 
		\begin{equation}
			\label{eq:deg-ft}
			d_{F_{t}}(v_{t}, S(v_1, \ldots, v_{t-1}) \cap Y_{t}) \geq 1.5 \delta \xi pn= 1.5 \delta^{t}pn\,.
		\end{equation}
		
		Let ${\ff}_{t} = p^{-1}\ind_{F_t}$. By Theorem~\ref{t:transference}, there is a function $\df_t: [n]^{(2)} \to [0,1]$ such that
		$$\|{\ff}_{t} - (1+\eps_1)\df _{t} \| \leq \eps_1\,.$$
		Let $\df_t '= \df_t  \! \restriction_{V_t \times ([n] \setminus (V_1 \cup \dots V_t))}$, and let $D_t$ be a graph sampled from $\df_t'$. By the same argument as in the induction basis, we may assume that
		$$\cutnorm{D_t - \df_t} \leq 7 \eps_1 \text{\quad and \quad } 
		\cutnorm{D_t - p^{-1}\ind_{F_t}} \leq 9\eps_1\,,$$
		where $D_t$ stands for the indicator function $\ind_{D_t}$.

		Let $\Hc_{t}$ consist of $t$-tuples $(v_1, \ldots, v_{t})$ such that $v_{t} \in N_{\Jn^{*}}(v_1, \ldots, v_{t-1})$ and
		\begin{equation} \label{eq:deg-v-D}
			d_{D_{t}}(v_{t}, S(v_1, \ldots, v_{t-1}) \cap Y_{t}) \geq \delta^{t} n\,.  
		\end{equation}
		Using Lemma~\ref{l:degree-cutnorm} and~\eqref{eq:deg-ft}, for each $(v_1, \ldots, v_{t-1}) \in \Hc_{t-1}'$, there are at most $\eps_2 n$ vertices $v_t \in N_{\Jn^{*}}(v_1, \ldots, v_{t-1})$ violating~\eqref{eq:deg-v-D} (recalling that $\eps_2 > \eps_1^4$),  so indeed
		\begin{align*}
			|\Hc_{t}| 
			&\geq |\Jn^{*}[E(\Hc_{t-1}'), V_t]| - \eps_2 n^{t}\\ 
			&\geq \frac {1}{100} |\Hc_{t-1}| \delta \xi^2 |V_t|\\
			& \geq   10^{-2t} \delta^{(t-1)^2 + 2t -1} |V_1|\ldots |V_t|\\ 
			&=  10^{-2t}\delta^{t^2} |V_1|\ldots |V_t|\,,
		\end{align*}
		where we used the inductive hypothesis in the second line.
		
		Finally, we set $S(v_1, \ldots v_{t})= N_{D_{t}}(v_{t}) \cap S(v_1, \ldots, v_{t-1}) \cap Y_{t}$ 
		to obtain a set which satisfies asserting~\ref{it:Sv} of Claim~\ref{c:hypergraphs-stronger}.
	\end{proof}

	For the remainder of the proof, we do not need properties~\ref{it:Di} and~\ref{it:Sv}, but only the following consequences. We remark that it is crucial that the relative density of $K_{|M|}$-copies mentioned in ~\ref{it:Sv-weaker} (denoted $\nu$) does not depend on $\alpha$, but only on $\delta$ and $\ell$. On the other hand $|V_i|/n$ may depend on $\alpha$.

	For every subset $M\subseteq [L]$ below we show
	\begin{enumerate}[label=\Alabel]
		\item \label{it:Sv-weaker} 
		The graph $\bigcup_{i \in M}D_i $ contains at least $\nu n\prod_{i \in M}|V_i|$ { copies of } $K_{|M|+1}. $		
		\item \label{it:Di-w} If $|M| \leq \ell$,  then
		\[
		p^{-\binom \l 2} \cdot \Lambda_{K_\ell} \bigg(\bigcup_{i \in M} F_{i}\bigg)
		\geq 
		\Lambda_{K_\ell} \bigg(\bigcup_{i \in M} D_{i}\bigg)- \eps_2\,.
		\]
	\end{enumerate}
	We first show that part~\ref{it:Sv} of Claim~\ref{c:hypergraphs-stronger} implies~\ref{it:Sv-weaker}. Fix any $(v_1, \ldots, v_ L) \in \Hc_ L$, and let $|V^*_{ L+1}| = S(v_1, \ldots, v_ L)$, so $|V^*_{ L+1}| \geq \delta^ L n$. Now, for each $v_{ L+1} \in V^*_{ L+1}$, the vertices $\{v_i\colon i \in M\} \cup \{v_{ L+1} \}$ form a clique in $\bigcup_{i \in M}D_i$, since $N_{D_i}(v_i)$ contains $v_j$ for $j>i$ by~\ref{it:Sv}. 
	The number of choices for $(v_i\colon i \in M) $ contained in some edge $(v_1, \ldots, v_ L) \in \Hc_ L$ is at least
	$$|\Hc_ L| \left(\prod_{i \in [ L] \setminus M}|V_i| \right)^{-1} \geq 10^{-2 L}\delta^{ L^2}\prod_{i \in M} |V_i|.$$
	Putting these two bounds together, we obtain at least $n \delta^ L \cdot 10^{-2 L}\delta^{ L^2}\prod_{i \in M} |V_i|$ copies of~$K_{|M|+1}$, which implies~\ref{it:Sv-weaker} since $L \leq 2 \ell^2$.
	
	Secondly, we claim that the $D_i$ satisfy~\ref{it:Di-w}. Let $D = \bigcup_{i \in M} D_i$, $\df = \sum_{i \in M} \df_i$, 
	and  $F = \bigcup_{i \in M} F_i$, so that $\ind_{F} = \sum_{i \in M} \ind_{F_i}$. By the triangle inequality and 
	part~\ref{it:Di} of Claim~\ref{c:hypergraphs-stronger}, we have
	$$\Big\| p^{-1}\ind_{F} - (1+\eps_1)\sum_{i \in M} \df_i \Big\|\leq \eps_1 \ell.$$
	Hence, by Theorem~\ref{t:transference} (applied with $\eps = \eps_1 \ell$),  and taking $\eps_1$ sufficiently small depending on $\eps_2$, we have
	$$\Lambda_{K_\l}(p^{-1} \ind_F) \geq \Lambda_{K_\l}(\df) - \eps_1 \ell^{3\ell^2} \geq \Lambda_{K_\l}(\df) - \frac{\eps_2}{2}. $$
	Moreover, $\cutnorm{\df_t - D_t} \leq 7\eps_1$ for $t \in [L]$, so using Lemma~\ref{l:cutnorm-copies}, we have
	$\Lambda_{K_\l}(\df)  \geq \Lambda_{K_\l}(D)- \frac{\eps_2}{2}. $  It follows that
	$$\Lambda_{K_\l}(p^{-1} \ind_F) \geq \Lambda_{K_\l}(D) - \eps_2\,,$$
	as required for the proof of~\ref{it:Di-w}.

	We now complete the proof of Lemma~\ref{l:strictly-min}. Let $I' \subset [L]$ be a set of order $L/2$ such that~$\diamond$ is constant on $I'$ and, without loss of generality, we may assume that $\diamond_i = <$ for $i \in I'$. Moreover, let $I \subset I'$ be a set of order $\ell-1$ 
	such that 
	\begin{enumerate}[label=\rmlabel]
		\item either $\psi_i \neq \star$ for $i \in I$ and $\psi$ is constant or injective on~$I$, 
		\item or $\psi_i=\star$ for $i \in I$.
	\end{enumerate}	
	Let
	$$\theta = \prod_{i \in I} \frac{\left|V_i\right|}{n}, \quad F = \bigcup_{i \in I} F_i \text{ \quad and \quad} D = \bigcup_{i \in I} D_i\,,$$ 
	and note that $\theta$ is bounded from below by a constant depending on $\alpha, \delta, \ell$, due to the lower bound on $|V_i|$
	in Claim~\ref{c:hypergraphs-stronger}. By assertion~\ref{it:Sv-weaker}, $D$ contains at least $\nu \theta n^\ell$ copies of $K_{\ell}$. 
	Hence, owing to~\ref{it:Di-w} and $ \eps_2 \leq \frac{1}{10} \nu \theta$, the graph~$F$  contains at least 
	$\frac 12 \nu \theta n^{\ell} p^{\binom \ell 2}$ copies of $K_{\ell}$.
	
	All these copies are non-strictly min-coloured by construction of $F$ (i.e., $\varphi(uv) = c(u) $ for $u< v$, $uv \in F$), and now we will use~\ref{it:psi}  and the choice of $I$ to show that there is actually a strictly min-coloured or a monochromatic copy.
	We first show that each copy of~$K_{\ell}$ in $F$ has exactly one vertex in $V_i$ for $i \in I$. Let $v_1 < v_2 < \dots < v_\ell$ be the vertex set of a $K_\ell$ in $F$, and recall the property~\ref{it:Fi} of Claim~\ref{c:hypergraphs-stronger} for~$F$. 
	Since all the edges in $F$ have the starting point in $\bigcup_{i \in I}V_i$, we have that  $\{v_1, \ldots, v_{\ell-1}\} \subseteq \bigcup_{i \in I}V_i$. But each $V_i$ is an independent set in $F$, so it contains at most one (and hence exactly one) vertex from $\{v_1, \ldots, v_{\ell -1}\}$.
	
	If $\psi_{i}\neq \star$ for $i \in I$, any copy of $K_\ell$ in $F$ is min-coloured (in case $\psi$ is injective on $I$) or  monochromatic (in case $\psi$ is constant on $I$); to see this, recall that by~\ref{it:psi} of Claim~\ref{c:hypergraphs-stronger}, if $uv \in E(F_i)$, $\varphi(uv) = \psi_i$.
	
	Suppose that $\psi_{i}=\star$ for $i \in I$. For $i, j \in I$, let $\Kc_{ij}$ be the collection of $K_{\ell}$-copies containing vertices $v_i \in V_{i}$ and $v_j \in V_{j}$ with $c(v_i) = c(v_j)$. We will  show that for all $i \neq j \in I$
	\begin{equation} 		\label{eq:vertex-conflict}
		|\Kc_{ij}| \leq 3\alpha \theta \npl.
	\end{equation} 
	(This follows easily from Theorem~\ref{t:transference1}~\ref{it:trans1-2} when each colour class in $V_i$ is of size $\alpha |V_i|$, but we need to be slightly more careful about smaller vertex classes.)
	
	Partition the colours in $c[V_i]$ into \textit{clusters} $1, \ldots, m$ with $m \leq 2\alpha^{-1}$ such that for each $k \in [m]$, the proportion of vertices in cluster $k$ in $V_{i}$ lies in $[\alpha, 2 \alpha]$. Note that such a partition exists since $|c[V_{i}]| \leq \alpha |V_{i}|$ by~\ref{it:psi} of Claim~\ref{c:hypergraphs-stronger}.
	For each cluster $k \in [m]$, let $\beta(k) $ (resp.~$\gamma(k)$) be the proportion of vertices $v$ in $V_{i}$ (resp.~$V_{j}$) such that $c(v)$ is in cluster $k$, so $\alpha \leq \beta(k) \leq 2 \alpha$. By Theorem~\ref{t:transference1}~\ref{it:trans1-2}, the number of $K_{\ell}$-copies with vertices in cluster $k$ in both $V_{i}$ and $V_{j}$ is at most
	$$ (\beta(k) \gamma(k)\theta + \eps_1) \npl \leq (2\alpha \gamma(k) \theta + \eps_1) \npl.$$
	Summing over $k \in [m]$, corresponding to clusters $1, \ldots, m$, and using $\sum_{k \in [m]} \gamma(k) \leq 1$, we obtain
	$$|\Kc_{ij}| \leq \sum_{k \in [m]}(2 \alpha \gamma(k)\theta+\eps_1) \npl \leq (2\alpha \theta + 2\alpha^{-1} \eps_1) \npl.$$
	Taking $\eps_1 \leq \frac 13 \alpha^2 \theta$ implies~\eqref{eq:vertex-conflict}.
	
	The  bound~\eqref{eq:vertex-conflict} holds for any $i, j$, so taking $\alpha < \nu / (8\ell^2)$, we obtain $$\left| \bigcup_{i < j \in I} \Kc_{ij} \right| \leq 3 \ell^2 \alpha \theta \npl \leq  \frac {3}{8} \nu \theta \npl.$$ Recalling that $F$ contains at least $\frac 12 \nu \theta \npl$ copies of $K_{\ell}$, it follows that there is a copy outside $\bigcup_{i< j \in I} \Kc_{ij} $, which is then strictly min-coloured.
	
	This completes the proof of Lemma~\ref{l:strictly-min}.\qed
	\begin{rem} \label{remark:clique}
		The fact that $K_\ell$ is a clique was only used to show that each copy of $K_\ell$ in~$F$ has at most one vertex in each $V_i$ for $i \in I$. In the concluding remarks, we will discuss to what extent our proof extends to general graphs $H$.
	\end{rem}

	\section{Concluding remarks}\label{sec:conclusion}
	\subsection{Thresholds for canonical Ramsey properties for general graphs}
	Recall that for an ordered graph $H$, we defined $\hat p_H$ as the threshold for the property $\Gnp \clra (H)$ and 
	Theorem~\ref{t:main} establishes $\hat p_{K_\l} = n^{-\frac{2}{\l+1}}$.  The problem of determining the threshold $\hat{p}_H$ for ordered graphs $H$ which are not complete is still open, but there are some partial results.
	
	Firstly, Alvarado, Kohayakawa, Morris, and Mota~\cite{morris23pers} studied a closely related problem for even cycles $C_{2\ell}$. Their result implies that for $p = C n^{-1/m_2(C_{2\ell})} \log n$, any colouring of~$\Gnp$ contains a canonical copy of the cycle $C_{2 \ell}$. However, in their work the ordering of the random graph $\Gnp$ is determined after the colouring.
	
	Secondly, for a strictly balanced graph $H$, our proof guarantees for $p\gg n^{-1/m_2(H)}$
	a canonical copy of $H$, but one cannot require a specific vertex ordering of~$H$.  
	This statement is shown using the following modification of Theorem~\ref{t:transference}, 
	which actually slightly simplifies the present proof of Lemma~\ref{l:strictly-min} for $K_\ell$ as well, but at the expense of introducing some additional formalism.
	For a collection of functions $f = (f_e)_{e \in H}$, define
	$$\Lambda^\dagger ( (f_e)_{e \in H}) =  \sum_{u_1, \ldots, u_\ell} \prod _{ij \in H} f_{ij}(u_i, u_j),$$
	that is, the density of $H$-copies in which the image of each edge $e \in H$ is weighted by $f_e$. 
	A small modification of Corollary 3.7 in~\cite{cg16} (which appears in~\cite{cgss14}*{page~17} in order to prove Theorem~\ref{t:transference1}) implies that if $\| \ff_e - \df_e\| = o(1)$ for $e \in E(H)$, then 
\[
	|\Lambda^\dagger ( (\ff_e)_{e \in H}) -\Lambda^\dagger ( (\df_e)_{e \in H}) | = o(1)\,.
\] 
Hence,  our proof can be carried out with the following modification. Assume that we have found our desired set of $\ell-1$ indices $I$, and that $\diamond_i = <$ for $i \in I$; we may relabel so that $I = [\ell-1]$. Then we can define $f_{ij} = F_i$ for $ij \in E(H)$ with $i<j$. Now, any  embedding $\beta \colon H \to G$ generated by the proof has the property that for $ij \in E(H)$ with $i<j$,  $\zeta(i)\zeta(j)$ lies in $F_i$, so its colour is determined by $\min(\zeta(i), \zeta(j))$.

    Returning to the issue of vertex ordering, when a pair $ij$ is \textit{not} an edge of $H$, the proof does not guarantee that $\zeta(i) < \zeta(j)$.
	
	\subsection{Canonical colourings in random hypergraphs} Furthermore, it would be interesting to investigate extensions of 
	Theorem~\ref{t:main} to $k$-uniform hypergraphs for $k\geq 3$. Namely, in their original work Erd\H os and Rado~\cite{er50} established a canonical Ramsey theorem for $k$-uniform hypergraphs. However, their proof for $k$-uniform hypergraphs 
	used Ramsey's theorem for $2k$-uniform hypergraphs and this seems to be an obstacle for transferring it 
	to random hypergraphs at the right threshold. Hence, for transferring their result to the random setting, it seems necessary to start with a proof which avoids the use of hypergraphs with larger uniformity. Such proofs can be found in~\cites{rrsss22,shelah95}.

	\appendix
	
	\origsection{Transference}  \label{app:transference}
    The purpose of this Appendix is to help the reader verify how Theorem~\ref{t:transference} follow from the proof of Conlon and Gowers~\cite{cg16}. For an informal discussion of the Conlon--Gowers approach, we also refer the interested reader to~\cite{cgss14}*{Section~3}.
    
	First we informally outline the proof of Theorem 9.3 from~\cite{cg16}, which corresponds to our  Theorem~\ref{t:transference1}. Then we state the formal claims that we need from~\cite{cg16}, and show how they are applied to deduce Theorem~\ref{t:transference}. The proof of this theorem is entirely contained in~\cite{cg16}, but some elements which we use (the norm $\norm{\cdot}$ and the dense model $\df$) are only defined within Theorem 4.5~in~\cite{cg16}. 
 
 The statements whose proofs we would like to expound are Theorem 9.3, and its corresponding deterministic result, Theorem 4.10. Unfortunately, these two proofs are not actually spelled out in~\cite{cg16}. Instead, the authors prove Theorem 9.1 and Theorem 4.5, which are Szemer\'edi-type results for random sets, and say that the proofs of Theorem 9.3 and Theorem 4.10 are `much the same'. Moreover, their setting is much more general -- they work with random subsets of a set $X$, which for us is just the set of edges of a complete graph $K_n$. In particular, for us, $|X| = \binom n2$. 
	
	Theorem 9.3 from~\cite{cg16} states that asymptotically almost surely, any subgraph $F$ of $\Gnp$ with $p = Cn^{-1/m_2(H)}$ can be approximated by a \emph{dense model} $D$ which asymptotically matches the edge distribution and the number of $H$-copies in $F$.
	
	As mentioned, instead of graphs, we work with functions from $[n]^{(2)}$ to $\R$, or weighted graphs. For a random graph $G\in\Gnp$, the \emph{associated measure} of $G$ is defined as $\mu = \mu_G =p^{-1}\ind_G$.  Given an $m$-tuple of functions 
	$\mu = (\mu_1, \ldots, \mu_m) \in \R^{[n]^{(2)}}$ (which will later be taken as the associated measures of $m$ independent copies of $\Gnpst$), Conlon and Gowers introduce the set of \textit{$(\mu, 1)$-basic anti-uniform functions} $\Phi_{\mu, 1}$, which have the key property that the number of $H$-copies in a weighted graph $\ff \leq m^{-1}(\mu_1 + \dots + \mu_m)$ can be bounded in terms of the inner products
	$$\max \left \{| \langle \ff, \phi \rangle|\colon \phi \in \Phi_{\mu, 1}  \right \}\,.$$

	The norm  $\norm{\cdot}$ is defined as
	$$\| \ff \| = \max \left \{| \langle \ff, \phi \rangle|\colon \phi \in \Phi_{\mu, 1}  \right \} \cup \{\| \ff \|_\square \}\,,$$
	where the term $\|\ff \|_\square$ is just appended to ensure that $\norm{\cdot}$ also controls the edge distribution of a weighted graph. This corresponds to Definitions 3.6 and 4.9 from~\cite{cg16}. 
	
	One caveat in this description is that $\norm{\cdot}$ only controls the number of $H$-copies under certain deterministic conditions on $\mu_1, \ldots, \mu_m$. In the context of graphs, these conditions, denoted (P0)--(P3$'$) in~\cite{cg16}*{Section~4}, imply the property that the corresponding random graphs have a sufficiently homogeneous edges distribution, which is a well-known necessary condition for all similar counting results in sparse random graphs.  To prove Theorem~\ref{t:transference1}, Conlon and Gowers show three statements. Firstly, asymptotically almost surely, the associated measures $\mu_1, \ldots, \mu_m$ of $\Gnpst$ with $\pst = Cn^{-1/m_2(H)}$ satisfy (P0)--(P3$'$).  Secondly, assuming (P0)--(P3$'$), for any $\ff \leq m^{-1}(\mu_1 + \dots + \mu_m)$ there is a \emph{dense model} $\df\colon [n]^{(2)} \to [0,1]$ with
	$$\|\ff - (1+\eps)\df \| \leq \eps\,.$$
	Thirdly, again assuming (P0)--(P3$'$), $\df$ is a useful approximation for $\ff$ in our context, since~$\norm{\cdot}$ controls the edge distribution and the number of $H$-copies in $\ff$ and $\df$. We have decided not to state properties (P0)--(P3$'$) since this would require reproducing large sections of~\cite{cg16} and introducing additional concepts.
	
	These three statements imply that asymptotically almost surely, any such function $\ff \leq m^{-1}(\mu_1 + \dots + \mu_m)$ has a suitable dense model. To reach the same conclusion for subgraphs of $\Gnp$, a small additional step is needed (cf.\ Proof of Theorem 9.1 in~\cite{cg16}). Given $m$ independent samples of $\Gnpst$ with $\pst = Cn^{-1/m_2(H)}$, let $G$ be the union of $U_1, U_2, \ldots, U_m$. Then $G$ is distributed as $\Gnp$ with $p = 1- (1-\pst)^m$, which is slightly smaller than $\pst m$. Thus the hypothesis $\ff \leq p^{-1}\ind_G$ does not quite imply that $\ff \leq m^{-1} \pst^{-1} (\ind_{U_1} + \dots + \ind_{U_m})= m^{-1}(\mu_1  \dots + \dots \mu_m)$. Still, since $p = \pst  m(1+o(1))$, this caveat can be resolved by slightly rescaling $\ff$, which we do at the start of the proof. We remark that this strategy of exposing $\Gnp$ in $m$ copies (for a large constant $m$) is used in~\cite{cg16} in order to be able to (define and) verify properties (P0)--(P3$'$).

	Now we formally state the claims which are used for deducing Theorem~\ref{t:transference}, and where they can be found in~\cite{cg16}. Say that $\mu_1, \ldots, \mu_m$ satisfy the property $\Pc(\eta, \lambda, d, m)$ if they satisfy properties (P0)--(P3$'$) stated in~\cite{cg16}*{Section 4}. 
	The following statement\footnote{Specifically, this claim is the first sentence of their proof of 9.1.} can be found in the proof of Theorem 9.1 in~\cite{cg16}. 
	\begin{lemma} \label{l:cg-gnp}
		Given $\eta, \lambda, d, m$, there is $C$ such that for $\pst  = Cn^{-1/m_2(H)}$ the following holds. If  $U_1, \ldots, U_m \in \Gnpst$ are mutually independent and $\mu_i = \pst^{-1}\ind_{U_i}$ is the associated measure of $U_i$ for $i \in [m]$, then $\mu_1, \ldots, \mu_m$ satisfy $\Pc(\eta, \lambda, d, m)$ asymptotically almost surely.
	\end{lemma}

	The following lemma can be deduced from the proof of Theorem 4.5 in~\cite{cg16}. (In their proof, the \textit{dense model} $\df'$ is denoted by $g$. The existence of $g$ is in the fourth sentence of the proof, and the final display on page 391 corresponds to (ii).) 
 \begin{lemma} ~\label{l:cg-deterministic}
		Given $\eps >0$, there are sufficiently small constants $\eta, \lambda >0$ and large integers $d, m$ such that if  $\mu_1, \ldots, \mu_m$ satisfy $\Pc(\eta, \lambda, d, m)$ and $f \leq m^{-1}(\mu_1 + \dots \mu_m)$, then the following holds.
		\begin{enumerate}[label=\rmlabel]
			\item \label{it:cg-exist} There is $\df'\colon [n]^{(2)}\to \R$ with $0 \leq \df' \leq 1$ and $\|f - (1+\eps/4) \df' \| \leq \frac{\eps}{2} $.
			\item \label{it:cg-H-count} If $\df \colon [n]^{(2)}\to \R$ is a function with $0 \leq \df \leq 1$ and $\|f - (1+\eps) \df \| \leq \eps$, then
			$$\Lambda_H(f) \geq \Lambda_H(\df) - 4 |E(H)|\cdot\eps.$$
		\end{enumerate}
	\end{lemma}

	Now we can deduce our desired result.
	\begin{proof}[Proof outline for Theorem~\ref{t:transference}]
		Given $\eps >0$, let $\eta , \lambda, d$ and $m$ be as required for the conclusion of Lemma~\ref{l:cg-deterministic} to hold. The random graph $G$ will be sampled  in $m$ rounds -- that is, we  set $\pst = Cn^{-1/m_2(H) }$ and $p = 1- (1-\pst )^m \geq \left(1- \eps/4 \right)\pst m$ for sufficiently large $n$.  Following the notation of Conlon and Gowers, let $U_1, \ldots, U_m$ be $m$ mutually independent random graphs sampled independently with edge probability $\pst$, where $C$ is a sufficiently large constant. Our random graph $G$ with edge probability $p$ will then be sampled by taking the union of $U_1, \dots, U_m$, which indeed has the claimed distribution. For $i \in [m]$, let $\mu_i = \pst ^{-1}\ind_{U_i}$ be the associated measure of $U_i$, and define $\mu =m^{-1}(\mu_1 + \dots + \mu_m)$. Assume that $(\mu_i)_{i \in [m]}$ satisfy the property $\Pc = \Pc(\eta, \lambda, d, m)$. By Lemma~\ref{l:cg-gnp}, this occurs asymptotically almost surely.

		We will apply Lemma~\ref{l:cg-deterministic} to deduce the existence of $\df_F$ and part~\ref{it:transf-Hcount}. Let $F$ be a subgraph of $G$, so $0 \leq \ind_F \leq \ind_G \leq \sum_{i \in [m]} \ind_{U_i}$. Define
		\begin{equation}  \label{eq:f-def}
			\tilde{\ff} = p^{-1} \ind_F  \text{\quad and \quad} \ff = \frac{1+\eps/4}{1+\eps} p^{-1} \ind_F.
		\end{equation}
		
		We claim that $\ff \leq \mu$. Indeed, recalling that  
        $p  \geq \left(1- \eps/4 \right)\pst m$, and hence 
        \[
            p^{-1} \cdot \frac{1+\eps/4}{1+\eps} 
            \leq 
            (m\pst)^{-1},
        \]
        for large $n$, we have
		$$\ff = \frac{1+\eps/4}{1+\eps} p^{-1} \ind_F \leq \frac{1+\eps/4}{1+\eps} p^{-1} \sum_{i \in [m]} \ind_{U_i}\leq (m \pst)^{-1}\sum_{i \in [m]} \ind_{U_i} = m^{-1}\sum_{i \in [m]}\mu_i = \mu\,.$$
		Thus we can apply the above-mentioned claims.
		
		Lemma~\ref{l:cg-deterministic}~\ref{it:cg-exist} applied to $f = \ff$ gives a function $\df_F=\df'$ such that $\|\ff - (1+\eps/4)\df_F \| \leq \frac{\eps}{2} $. Multiplying by $\frac{1+\eps}{1+\eps/4}$ and recalling~\eqref{eq:f-def}, we obtain
		$$\| \tilde{\ff} - (1+\eps)\df_F \| \leq {\eps}\,,$$
		as required.
		
		To see~\ref{it:transf-Hcount}, take  $\df$ with $\|p^{-1} \ind_F - (1+\eps) \df\|\leq \eps$ and $0 \leq \df \leq 2\ell$. We have
		$$\norm{ \frac{\ind_F}{2 \ell p} - \frac{(1+\eps)\df}{2\ell} } \leq \frac{\eps}{2\ell} < \eps\,.$$
		We may apply Lemma~\ref{l:cg-deterministic}~\ref{it:cg-H-count} with $f = \frac{\ind_F}{2\ell p} \leq \mu$ and $\df$ replaced by $\frac{\df}{2\ell} \leq 1$ to obtain
		$$\Lambda_H \left(\frac{\ind_F}{2\ell p} \right) \geq 
		\Lambda_H \left(\frac{\df}{2\ell } \right) -  4 |E(H)|\cdot\eps\,.$$
		Using the fact that $\Lambda_H(\alpha \ff')  = \alpha^{|E(H)|}\Lambda_H(\ff')$ for any constant $\alpha \geq 0$ and any $\ff'\colon [n]^{(2)}\to \R$, it follows that
		$$p^{-|H|} \Lambda_H(F) \geq \Lambda_H(\df)  - 4\eps|H| (2\ell)^{|E(H)|} \geq \Lambda_H(\df)  - 4\eps (2\ell)^{\ell^2}, $$
		as required.
		
		Statement~\ref{it:transf-edgedist} follows from $\df \leq 1$, the definition of $\cutnorm{\,\cdot\,}$
		and the triangle inequality. That is,
		\[
		\cutnorm{\ff - \df} 
		\leq 
		\cutnorm{\ff - (1+ \eps) \df} + \cutnorm{\eps \df} 
		\leq  
		\cutnorm{\ff - (1+ \eps) \df} + \eps 
		\leq 2 \eps\,.\qedhere
		\]
	\end{proof}

 \subsection*{Acknowledgements}
    We would like to thank Henri Ortm\"uller for pointing out reference~\cite{ajmp03} to us and we are also grateful to 
    the anonymous referees for their helpful comments.

\end{document}